\documentclass{amsart}
\usepackage{lmodern}
\usepackage{stmaryrd}
\usepackage{lipsum}
\usepackage{subfigure}
\usepackage{hyperref}
\usepackage{graphicx,color}
\usepackage{stix} 
\usepackage{amssymb,amsmath,bm,mathtools}
\usepackage{enumerate}
\usepackage{tkz-graph}
\usepackage{ragged2e}
\usepackage{enumitem}
\usepackage{adjustbox}
\usepackage{tipa}
\usepackage{circuitikz} 
\usepackage{scalerel}
\usepackage[margin=1in]{geometry}
\usepackage{parskip}
\usepackage{comment}

\usepackage{tikz-cd}
\usepackage{adjustbox}
\usepackage{tikz}
\usepackage[all,knot]{xy}
\usepackage{transparent}
\usepackage{graphicx,import}

\usetikzlibrary{arrows,decorations.pathmorphing}
\usetikzlibrary{backgrounds,positioning}
\usetikzlibrary{fit,petri,shapes.misc}
\usetikzlibrary{shapes.geometric}
\usetikzlibrary{decorations.markings}
\usetikzlibrary{calc}
\usetikzlibrary{braids}

\newlength{\defaultpgflinewidth}
\tikzset{auto}

\tikzset{empty/.style={circle,inner sep=0pt,minimum size=6mm}}
\tikzset{emptyvt/.style={circle,inner sep=0pt,minimum size=0mm}}

\tikzset{plain/.style={circle,draw,very thick,
inner sep=0pt,minimum size=6mm}}

\tikzset{xplain/.style={circle,draw,very thick,
inner sep=0pt,minimum size=8mm}}

\tikzset{smallplain/.style={circle,draw,very thick,
inner sep=0pt,minimum size=4mm}}

\tikzset{xsplain/.style={circle,draw, thick,
inner sep=0pt,minimum size=1.5mm}}

\tikzset{tinyplain/.style={circle,draw, thick,
inner sep=0pt,minimum size=1mm}}

\tikzset{smalldotted/.style={circle,draw,very thick, densely dotted,
inner sep=0pt,minimum size=3mm}}

\tikzset{dottedplain/.style={circle,draw,very thick, densely dotted,
inner sep=0pt,minimum size=6mm}}

\tikzset{normaldot/.style={circle,black,fill=black,inner sep=0pt,minimum size=2mm}}

\tikzset{rectplain/.style={rectangle,draw,very thick,minimum size=6mm}}
\tikzset{bigplain/.style={rectangle,draw,very thick,minimum size=1cm}}

\tikzset{triangular/.style={regular polygon, regular polygon sides=3, draw,very thick,
inner sep=0pt,minimum size=1.2cm}}

\tikzset{arrow/.style={->,thick}}
\tikzset{dashedarrow/.style={->,dashed,thick}}
\tikzset{dottedarrow/.style={->,dotted,thick}}
\tikzset{mapto/.style={|->,thick}}
\tikzset{->-/.style={decoration={markings, mark=at position #1 with {\arrow{>}}},postaction={decorate}}}

\tikzset{implies/.style={thick,double,double equal sign distance,-implies}} 

\tikzset{line/.style={thick}}
\tikzset{dottedline/.style={dotted,thick}}
\tikzset{dashedline/.style={dashed,thick}}

\tikzset{inputleg/.style={<-,thick}}
\tikzset{outputleg/.style={->,thick}}
\tikzset{dottedinput/.style={<-,dotted,thick}}

\usepackage{transparent}
\usepackage{graphicx,import}
\usepackage{caption}

\usepackage{enumitem}

\hypersetup{
	colorlinks   = true, 
	urlcolor     = blue, 
	linkcolor    = black, 
	citecolor   = brown 
}

\usepackage{todonotes}
\newcommand\reallywidehat[1]{\arraycolsep=0pt\relax%
\begin{array}{c}
\stretchto{
  \scaleto{
    \scalerel*[\widthof{\ensuremath{#1}}]{\kern-.5pt\bigwedge\kern-.5pt}
    {\rule[-\textheight/2]{1ex}{\textheight}} 
  }{\textheight} %
}{0.5ex}\\           
#1\\                 
\rule{-1ex}{0ex}
\end{array}
}

\newcommand{\bC}{\mathbf{C}}

\newcommand{\bT}{\mathbf{T}}
\newcommand{\bE}{\mathbf{E}}

\newcommand{\bA}{\mathbf{A}}

\newcommand{\calO}{\mathcal{O}}
\newcommand{\calQ}{\mathcal{Q}}

\newcommand{\bS}{\mathbf{S}}

\newcommand{\kk}{\mathbb{K}}    

\newcommand{\PaA}{\mathsf{PaA}}
\newcommand{\CoRB}{\mathsf{CoRB}}
\newcommand{\PaRB}{\mathsf{PaRB}}
\newcommand{\PaRCD}{\mathsf{PaRCD}}
\newcommand{\RCD}{\mathsf{RCD}}
\newcommand{\PaCD}{\mathsf{PaCD}}
 
\newcommand{\A}{\mathsf{A}} 

\newcommand{\magma}{\Omega}

\newcommand{\hA}{{\A}_{\kk}}

\newcommand{\hRCD}{{\RCD}_{\kk}}
\newcommand{\hPaB}{{\PaB}_{\kk}}
\newcommand{\hPaA}{{\PaA}_{\kk}}
\newcommand{\hPaCD}{{\PaCD}_{\kk}}

\newcommand{\hPaRB}{{\PaRB}_{\kk}}
\newcommand{\hPaRCD}{{\PaRCD}_{\kk}}

\newcommand{\PaB}{\mathsf{PaB}}

\newcommand{\calP}{\mathcal{P}}

\newcommand{\cyc}{\mathbf{cyc}}

\newcommand{\Cyc}{\mathbf{Cyc}}
\newcommand{\Op}{\mathbf{Op}}

\newcommand{\Prop}{\mathbf{Prop}}

\newcommand{\id}{\operatorname{id}}

\newcommand{\Iso}{\operatorname{Iso}}

\newcommand{\End}{\operatorname{End}}
\newcommand{\Hom}{\operatorname{Hom}}

\newcommand{\ob}{\operatorname{ob}}
\newcommand{\Aut}{\operatorname{Aut}}

\newcommand{\HoAut}{\operatorname{HoAut}}
\newcommand{\Env}{\operatorname{Env}}

\newcommand{\lrar}{\ensuremath{\longrightarrow}}

\newcommand{\GT}{{\mathsf{GT}}}
\newcommand{\GRT}{{\mathsf{GRT}}}

\newcommand{\Assoc}{\mathsf{Assoc}}
\newcommand{\assoc}{\mathfrak{assoc}} 

\newcommand{\PB}{\mathsf{PB}} 

\newcommand{\RB}{\mathsf{RB}} 
\newcommand{\PRB}{\mathsf{PRB}} 

\newcommand{\ft}{\mathfrak{t}}
\newcommand{\ff}{\mathfrak{f}}
\newcommand{\fft}{\mathfrak{ft}}

\newtheorem{theorem}{Theorem}[section]
\newtheorem{thm}[theorem]{Theorem}
\newtheorem{prop}[theorem]{Proposition}
\newtheorem{lemma}[theorem]{Lemma}
\newtheorem{cor}[theorem]{Corollary}
\newtheorem*{thm*}{Theorem}

\newtheorem{Th}{Theorem}

\theoremstyle{definition}
\newtheorem{definition}[theorem]{Definition}
\newtheorem{remark}[theorem]{Remark}

\numberwithin{equation}{section}

\let\oldtocsection=\tocsection
\let\oldtocsubsection=\tocsubsection
\renewcommand{\tocsection}[2]{\hspace{0em}\oldtocsection{#1}{#2}}
\renewcommand{\tocsubsection}[2]{\hspace{1em}\oldtocsubsection{#1}{#2}}
\DeclareRobustCommand{\SkipTocEntry}[5]{}
\renewcommand{\paragraph}[1]{\textbf{{#1}.}\hspace{5pt}}

\title[Cyclic Symmetries of Chord Diagrams]{Cyclic Symmetries of Chord Diagrams}

\author[C. Singh]{Chandan Singh}
\address{School of Mathematics and Statistics \\ The University of Melbourne \\ Melbourne, Victoria, Australia}
\email{chandans@student.unimelb.edu.au}

\begin{document}

\begin{abstract}
We give a direct proof that the proalgebraic graded Grothendieck--Teichmüller group $\GRT_{\kk}$ is isomorphic to the group of automorphisms of the prounipotent cyclic operad of parenthesized ribbon chord diagrams based on Furusho’s 5-cycle reformulation of the pentagon equation. As an application, we describe a $\GRT_{\kk}$-action on the category of framed chord diagrams with self-dual objects, which is closely related to the target category of the Kontsevich integral for framed tangles.

\end{abstract}
\maketitle

\section*{Introduction}

The framed little disks operad carries a cyclic structure \cite{Budney}. This structure is inherited by the operad of parenthesized ribbon braids $\PaRB$, as well as by its infinitesimal counterpart, the operad of parenthesized ribbon chord diagrams $\PaRCD$. Since these operads also admit actions of the Grothendieck--Teichm\"uller group, $\GT_{\kk}$, and its graded version, $\GRT_{\kk}$, it is natural to ask how these Grothendieck--Teichm\"uller symmetries interact with their cyclic structures.

In this paper, we study the cyclic operad of parenthesized ribbon chord diagrams $ \PaRCD^{\cyc} $, obtained from the cyclic structure on framed Drinfeld--Kohno Lie algebras from \cite{campos2019configuration, severa2009formality}, and show that its automorphism group recovers $\GRT_{\kk} $. More precisely, our first main result (Theorem~\ref{thm: GRT is cyclic PaRCD}) identifies $ \GRT_{\kk} $ with the group of object-fixing automorphisms of the prounipotent cyclic operad $ \PaRCD^{\cyc} $:
\begin{Th}\label{thmA}
    $\GRT_{\kk} \cong \Aut_{\Cyc}^{+}(\hPaRCD^{\cyc})$.
\end{Th}
This gives a cyclic operadic characterization of $ \GRT_{\kk} $, analogous to the corresponding result for $ \GT_{\kk}$ by Robertson and the author in \cite{RS2025cyclicgt}. A closely related result of Willwacher \cite{willwacher2024cyclic} characterizes $\GRT_{\kk}$ in terms of homotopy automorphisms of Batalin-Vilkovisky cooperad seen as a cyclic dg Hopf cooperad. Theorem~\ref{thmA} should be seen as a strict $0$-truncation of Willwacher's result \cite[Corollary 1.5]{willwacher2024cyclic}. We work at the level of discrete automorphisms rather than homotopy automorphisms, giving a direct combinatorial proof that depends only on the defining equations of $\GRT_{\kk}$
and the cyclic structure of $\PaRCD$-- without passing through the homotopy theory of cooperads.

A key point is that the proof is genuinely different from the braid-side argument. In the case of $ \GT_{\kk} $, one can exploit the presentation of the operad $ \PaRB $ to verify the cyclic compatibility. For $ \PaCD $, no comparable finite presentation is known, so that strategy is unavailable, see, e.g. \cite[Remark 2.24]{calaque2024associatorsoperadicpointview} or \cite[Section 10.2]{FresseBook1}. Instead, we work directly with the defining pentagon and hexagon equations of $ \GRT_{\kk} $, together with the cyclic structure that comes from the spherical presentation of the framed Drinfeld--Kohno Lie algebras. This relies on Furusho's equivalent 5-cycle description of the pentagon equation of $\GRT_{\kk}$ from \cite{Furusho_pentagon}. In this way, the cyclic characterization of $ \GRT_{\kk} $ is obtained directly, without passing through associators. 

To place this theorem in context, we note that Drinfeld associators are in bijection with operad isomorphisms $ \hPaRB \to \hPaRCD $. The operad of parenthesized ribbon braids $\PaRB$ carries a cyclic structure inherited from the framed little disks operad (cf. \cite{campos2019configuration}), while the framed infinitesimal analog $\PaRCD$ acquires a cyclic structure from the framed Drinfeld--Kohno Lie algebras (cf.\cite{severa2009formality}, \cite[Section 5]{willwacher2024cyclic}). Proposition~\ref{prop: bijection of cyclic associators} shows that the classical operadic identification of associators lifts automatically to this cyclic framed setting: every associator determines a unique object-fixing cyclic operad isomorphism
\[
\hPaRB^{\cyc}\longrightarrow \hPaRCD^{\cyc}.
\]
This is a lifting of \cite[Lemma 3.2]{campos2019configuration} to the parenthesized setting. As a consequence, the bitorsor triple $( \GT_{\kk}, \Assoc_{\kk}, \GRT_{\kk} )$ admits a cyclic operadic refinement (Proposition~\ref{prop: bitorsor triple}) and we have a bitorsor isomorphism 
\[(\GT_{\kk},\Assoc_{\kk},\GRT_{\kk}) \leftrightarrow (\Aut^{+}(\hPaRB^{\cyc}), \text{Iso}^{+}(\hPaRB^{\cyc}, \hPaRCD^{\cyc}), \Aut^{+}(\hPaRCD^{\cyc})).\] In forthcoming work, Naef and Ševera explain how this cyclic operadic framework applies to the classical KZ associator: the framed KZ associator, obtained from the holonomy of the KZ connection along a path from $0$ to $1$, can be realized as a morphism of cyclic operads \cite{Drin,Severa26_cyclicKZ}.

This lift should be viewed as complementary to the main theorem rather than as its proof: it shows that the familiar associator correspondence is compatible with cyclic framing, while the identification
\[
\GRT_{\kk}\cong \Aut_{\Cyc}^{+}(\hPaRCD^{\cyc})
\]
requires a separate direct argument. Together, these results show that the cyclic framed operads $ \PaRB^{\cyc} $ and $ \PaRCD^{\cyc} $ provide natural homes for the two Grothendieck--Teichmüller groups and the associator torsor connecting them. 

The cyclic viewpoint also has a concrete application to chord diagram categories with duality. Hinich and Vaintrob \cite{hv_cyclic} showed that a cyclic operad satisfying the appropriate invariance condition gives rise, via the metric prop construction, to a tensor category in which duality morphisms are formally adjoined. Applying this construction to $ \PaRCD^{\cyc}_{\kk} $, we recover the category $\mathbf{A}'(\kk)$ of parenthesized chord diagrams with self-dual objects (Lemma~\ref{lemma: metric prop of PaRCD}). This category extends the classical chord diagram category $ \mathbf{A}(\kk) $, familiar from finite-type invariants and Kontsevich's universal knot invariant (see e.g. Bar-Natan~\cite{bar-natan_fundamental_1997,bar-natan_wheels_1997} and Le Murakami \cite{Le_Murakami_1996}), by imposing the additional strict self-duality relation $ (+)^* = + $. Our second main result (Theorem~\ref{prop: GRT action on metric prop}) states that the cyclic operadic description of $ \GRT_{\kk} $ therefore induces an action on this category of chord diagrams with self-dual objects:

\begin{Th}
The $\GRT_{\kk}$-action on the cyclic operad $\PaRCD^{\cyc}$ extends to the category $\bA'(\kk)$.
\end{Th} 

In this sense, the cyclic structure on $ \PaRCD $ provides the bridge from the operadic description of $ \GRT_{\kk} $ to a natural category of chord-diagrams with duals. The result may be viewed as the chord-diagram shadow of the $\GT_{\kk}$-action on tangles from \cite{RS2025cyclicgt} under the Kontsevich isomorphism . It is also closely related to Furusho's construction \cite{furusho_galois_2020} of a $ \GRT_{\kk} $-action on chord diagrams via the ABC-construction, although the precise relationship between these two actions remains to be understood (Remark~\ref{remark: two GRT actions}).

\subsection*{Organization of the paper} In Section~1 we introduce the cyclic operads $ \RCD^{\cyc} $ and $ \PaRCD^{\cyc} $. In Section~2 we prove that every Drinfeld associator lifts uniquely to a cyclic operad isomorphism $ \hPaRB^{\cyc}\to \hPaRCD^{\cyc} $. Section~3 contains the direct proof of the cyclic operadic characterization of $ \GRT_{\kk} $. Finally, in Section~4 we recall infinitesimal symmetric monoidal categories, identify $ \mathbf{A}'(\kk) $ with the metric prop generated by $ \PaRCD^{\cyc}_{\kk} $, and deduce the induced $ \GRT_{\kk} $-action on this category.

\subsection*{Conventions} Throughout we fix a field $\kk$ of characteristic zero. All morphisms are drawn from bottom to top.

\subsection*{Acknowledgements} 
We thank Hidekazu Furusho and Marcy Robertson for comments on an earlier draft and Zsuzsanna Dancso for useful discussions. We acknowledge the support of the Australian Government Research Training Program (RTP) Scholarship, the Australian Research Council Future Fellowship FT210100256, and the Dr Albert Shimmins Fund.

\section{A cyclic operad of chord diagrams }\label{sec: cyclic chord diagrams}

\subsection{Cyclic Operads} 

Let $\bE=(\bE,\otimes,\mathbb{1})$ be a symmetric monoidal category whose tensor product commutes with colimits. Let $\Sigma_n^+ = \Aut(\{0,1,\ldots,n\})$ denote the symmetric group on $n+1$ letters. We identify $\Sigma_n$ with the subgroup of $\Sigma_n^+$ consisting of permutations that fix $0$. We write $z_{n+1}$ for the cyclic permutation
\[
z_{n+1}(i) = i+1 \pmod{n+1}.
\]

A \emph{symmetric sequence} in $\bE$ is a sequence $\{\calO(n)\}_{n\geq 0}$, where each $\calO(n)$ is equipped with a right $\Sigma_n$-action. An \emph{operad} $\calO$ in $\bE$ is a symmetric sequence $\calO=\{\calO(n)\}_{n\geq 0}$ together with a distinguished operation $1\in \calO(1)$, called the unit, and a family of equivariant, associative, and unital partial compositions
\[
\circ_i : \calO(n)\times \calO(m)\to \calO(n+m-1),
\]
where $1\leq i\leq n$.

A morphism of operads $f:\calP\lrar\calQ$ is a morphism of the underlying symmetric sequences that commutes with the operad structure. An $\calO$-algebra is an operad morphism $\rho:\calO\to \End(A)$, where
\[
\End(A)=\{\End(A)(n)\}_{n\geq 0}=\{\Hom(A^{\otimes n},A)\}_{n\geq 0}
\]
is the endomorphism operad, see, e.g. \cite{gk_cyc}. Equivalently, an $\calO$-algebra consists of an object $A\in\bE$ together with a family of $\Sigma_n$-equivariant morphisms $\rho_n:\calO(n)\otimes A^{\otimes n}\rightarrow A$. 



\begin{definition}\label{def: cyclic operad}
A \emph{cyclic operad} is an operad $\calO=\{\calO(n)\}_{n\geq 1}$ equipped with action maps \[\begin{tikzcd}\calO(n)\times\Sigma_{n}^{+}\arrow[r, ""]& \calO(n), \end{tikzcd}\] extending the $\Sigma_n$-action, $n\geq 1$, such that for every $x\in\calO(n)$ and $y\in\calO(m)$, the operadic composition is compatible with the cyclic action as follows: 
\begin{equation} ~\label{cylic_formula}
      z_{n+m}^*(x\circ_i y)  =  \begin{cases} z_{n+1}^*(x)\ \circ_{i-1}\ y \ \ \text{if} \ \ 2\leq i\leq n \\
      z_{m+1}^*(y) \ \circ_m \ z_{n+1}^*(x) \ \ \text{if} \ \ i=1 \ \text{and} \ n\not=0. 
      \end{cases}
\end{equation}   
\end{definition}
A cyclic operad morphism $f: \calP \lrar \calQ$ is a morphism of the underlying operad that commutes with the cyclic structure. 

An $\calO^{\cyc}$-algebra in $\bE$ is an $\calO$-algebra $A$ equipped with a nondegenerate symmetric form
\[
d:A\otimes A\to \mathbb{1},
\]
such that for every $n\geq 1$, the morphism
\[
\calO(n)\otimes A^{\otimes (n+1)}
\xrightarrow{\rho_n\otimes \id_A}
A\otimes A
\xrightarrow{d}
\mathbb{1}
\]
is $\Sigma_n^+$-equivariant.

\subsection{Cyclic Structure on Ribbon Chord Diagrams}\label{subsec: Ribbon Chord Diagram}
For a group \(G\), its prounipotent completion is controlled by an associated Malcev Lie algebra, equivalently by the Lie algebra of primitive elements in the completed group Hopf algebra; see, for example, \cite[Section 8.3]{FresseBook1}. A basic example is provided by the pure braid group: the Drinfeld--Kohno Lie algebra $\mathfrak{t}_n$ is the Lie algebra associated with the prounipotent completion of $\PB_n$, and Kohno's isomorphism gives $(\PB_n)_{\kk} \cong \exp(\mathfrak{t}_n).$  The framed analog is obtained from the pure ribbon braid group $\PRB_n$: its associated Lie algebra is the framed Drinfeld--Kohno Lie algebra $\mathfrak{ft}_n$, and one has
\[
(\PRB_n)_{\kk} \cong \exp(\mathfrak{ft}_n).
\]
\begin{definition}\label{def:DK Lie algebra}
The \emph{framed Drinfeld-Kohno} Lie algebra $\mathfrak{ft}_n$ is a degree-completed free Lie algebra generated by symbols $\{t_{ij} = t_{ji}, 1\leq i,  j\leq n\}$, with relations 
\begin{equation}\label{equ:drinfeldkohno_rel}
	\begin{aligned}
	[t_{ij}, t_{kl}] &= 0 \quad \text{for $\{i,j\}\cap\{k,l\}=\emptyset$}, \\
	[t_{ij},t_{ki}+t_{kj}] &= 0 \quad \text{for distinct $i,j,k$}, \\
        [t_{ii},t_{jk}] &= 0 \quad \text{for any $i,j,k$}. \\ 
	\end{aligned}
\end{equation}
\end{definition}

There is a short exact sequence of Lie algebras $0 \lrar \kk^{n} \lrar \fft_{n} \lrar \ft_{n} \lrar 0$, where $\kk$ is the ground field, the map $\pi: \fft_{n} \lrar \ft_{n}$ given by $\pi(t_{ij}) = t_{ij}$, for $i \neq j$, and $\pi(t_{ii}) = 0$, for all $i$. The kernel of $\pi$ is essentially the Lie algebra generated by $t_{11}, t_{22}, \dots. t_{nn}$, which is isomorphic to $\kk^{n}$. This gives the  decomposition $ \fft_{n} = \bigoplus_{i=1}^{n} \kk t_{ii} \oplus \ft_{n}$. 

The Lie algebras $\mathfrak{ft}=\{\mathfrak{ft}_n\}_{n\geq 0}$ form an operad in completed Lie algebras (cf. \cite[Section 1.3]{severa2009formality}) in which the operadic composition \[\circ_k : \mathfrak{ft}_m \oplus \mathfrak{ft}_n \lrar \mathfrak{ft}_{m+n-1}\] is defined by: 
\begin{center}
    $0 \circ_k t_{ij} \mapsto t_{i+k-1 j+k-1} \quad \text{for all } k$,
\end{center}
\begin{equation*}
\begin{matrix} 
t_{ij} \circ_k 0 &\mapsto
\begin{cases}
t_{i+n-1 j+n-1 } & \text{if } k<i<j, \\
\sum_{p = i}^{i+n-1} t_{pj+n-1} & \text{if } k =i <j,\\
t_{i j+n-1} & \text{if } i < k <j,\\ 
\sum_{q = j}^{j+n-1} t_{iq} & \text{if } i < k =j,\\ 
t_{ij} & \text{if } i < j < k.
\end{cases} \quad \text{and} \quad
t_{ii} \circ_k 0 &\mapsto 
\begin{cases}
    t_{i+n-1 i+n-1 } & \text{if } k< i,\\ 
    \sum_{p = i}^{i+n-1} t_{pp} + \sum_{\{p,q\}\subset m} t_{pq} & \text{if } k =i,\\
    t_{ii} & \text{if } i < k.\\ 
\end{cases}.
\end{matrix} 
\end{equation*}

The Lie algebras \(\mathfrak{ft}_n\) assemble into a cyclic operad in pronilpotent Lie algebras; see \cite[Section~5.1]{willwacher2024cyclic}. Applying the completed universal enveloping algebra construction levelwise yields a functor \[\begin{tikzcd} \Op(\mathfrak{lie})\arrow[r, "\hat{U}(-)"] & \Op(\operatorname{Hopf})\end{tikzcd}\] from cyclic operads in pronilpotent Lie algebras to cyclic operads in complete Hopf algebras.

For each $n\geq 1$, the pronilpotent Lie algebra $\mathfrak{ft}_n$ has a completed universal enveloping algebra $\hat U(\mathfrak{ft}_n)$, which is a complete Hopf algebra. Elements of $\hat U(\mathfrak{ft}_n)$ may be viewed as formal linear combinations of polynomials in the generators $t_{ij}$, $1\leq i,j\leq n$. The exponential and logarithm define inverse bijections between the primitive and group-like elements of $\hat U(\mathfrak{ft}_n)$; see \cite[Proposition 8.1.5]{FresseBook1}. In particular, the exponential identifies $\mathfrak{ft}_n$ with the group $G(\hat U(\mathfrak{ft}_n))$ of group-like elements of $\hat U(\mathfrak{ft}_n)$.

\begin{definition}
For each $n\geq 1$, let $\hRCD(n)$ denote the groupoid with a single object and
\[
\Hom_{\hRCD(n)}(*,*)=\exp(\mathfrak{ft}_n)=G(\hat U(\mathfrak{ft}_n)).
\]
\end{definition}

The symmetric group $\Sigma_n$ acts on $\mathfrak{ft}_n$ by permuting the indices of the generators $t_{ij}$, and hence acts on $\hat U(\mathfrak{ft}_n)$ and on $\exp(\mathfrak{ft}_n)$. Explicitly, for $\sigma\in \Sigma_n$, we set
\[
\sigma^*(t_{ij}t_{lk}) = t_{\sigma(i)\sigma(j)}\, t_{\sigma(l)\sigma(k)}.
\]
The operadic partial compositions are induced from those on the Lie operad $\{\mathfrak{ft}_n\}_{n\geq 1}$: for $x\in \mathfrak{ft}_n$ and $y\in \mathfrak{ft}_m$, we define
\[
e^x\circ_i e^y := e^{x\circ_i y}.
\]
In this way, the collection $\hRCD=\{\hRCD(n)\}_{n\geq 1}$ forms an operad in complete Hopf algebras.


Following \cite[Section 3.1]{campos2019configuration}, the cyclic structure on $\hRCD$ can be completely determined by defining the following action of the $(01)$-transposition on $\hat{U}(\mathfrak{ft}_n)$: 
\begin{equation}\label{cyclic structure on RCD}
(01)^*(t_{ij})=\begin{cases} t_{ij} \ \text{if} \ i\not=1 \ \text{and} \ j\not=1;\\ -\sum_{k=1}^n t_{kj} \ \text{if} \ i=1 \ \text{and} \ j\not=1;\\
\sum_{k,l=1}^n t_{kl} \ \text{if} \ i=1 \ \text{and} \ j=1.\end{cases}
\end{equation}

\begin{remark}\label{remark: cyc framed DK}

The Lie algebra $\mathfrak{ft}_n$ admits a cyclic presentation by formally adding generators $t_{00}$, $t_{0i}=t_{i0}$, $i=1,\dots,n$ and rewriting the relations in a cyclically invariant form: 
\begin{equation*}
    t_{ij} = t_{ji}, \qquad [t_{ij}, t_{kl}] = 0 \quad \text{for $\{i,j\}\cap\{k,l\}=\emptyset$, \quad and } 
\end{equation*}
\begin{equation}\label{eq:residue relation}
    \sum_{i=0}^n t_{ij} = 0\, \quad \text{for each $j$} 
\end{equation}
One can check that this is equivalent to the cyclic structure described in \eqref{cyclic structure on RCD} by eliminating the generators $t_{ii}$ for $i=0,\dots, n$ using \eqref{eq:residue relation}, see~\cite[Section 5.1]{willwacher2024cyclic}. This cyclic presentation corresponds to the graded Lie algebra grLie($\PB_{n+1}(\mathbb{S^2})$) associated with the pure ribbon braid group on the sphere. Moreover, $\PB_{n+1}(\mathbb{S^2}) \cong \PB_{n+1}(\mathbb{C})/(\beta_1\beta_2 \cdots \beta^2_{n-1} \cdots \beta_2 \beta_1)$ , where $\beta_i$ are Artin braid generators. The quotient relation translates to the spherical residue condition \eqref{eq:residue relation}. 

\end{remark}

The magma operad $\magma(n)$ is a free symmetric operad in the category of sets generated by a binary operation $\mu(x_1,x_2) \in \Omega(2)$, and here $\Sigma_2$ acts freely on $\mu_2$. Elements of $\magma(n)$ are maximally parenthesized permutations of $n$ elements in $\Sigma_n$; for example, $((13)(24))5 \in \magma(5)$. The operadic composition is given by the substitution of letters such as $(1(23)) \circ_3 (21) = (1(2(43)))$.

\begin{definition}\label{def: PaRCD}
For each $n\geq 1$, we define a sequence of prounipotent groupoids $\hPaRCD(n)$ with: 
\begin{itemize}
    \item objects $\ob(\hPaRCD(n))=\magma(n)$; 
    \item morphisms defined by $\Hom_{\hPaRCD(n)}(p,q) = \exp(\mathfrak{ft}_n)$.
\end{itemize} 
The categorical composition in $\hPaRCD(n)$ is given by the multiplication in $\exp(\mathfrak{ft}_n)$.
\end{definition}

The collection $\hPaRCD=\{\hPaRCD(n)\}_{n\geq 1}$ forms an operad in prounipotent groupoids. On objects, the operadic composition is induced by that of $\magma$. On morphisms, it is given by the operad structure on $\RCD$. 

Every morphism in $\hPaRCD$ can be obtained from the following generating morphisms by categorical and operadic composition of the following generating morphisms (see Figure~\ref{fig:Generators PaRCD}):
\[
X_{1,2} \in \Hom_{\PaRCD(2)}((12),(21)), \qquad
H_{1,2} \in \Hom_{\PaRCD(2)}((12),(12)),
\]
\[
I_1 \in \Hom_{\PaRCD(1)}(1,1), \qquad \text{and} \quad
A_{1,2,3} \in \Hom_{\PaRCD(3)}((12)3,1(23)).
\]
This is analogous to Step~1 of \cite[Theorem 10.3.4]{FresseBook1}.

\begin{figure}[ht]
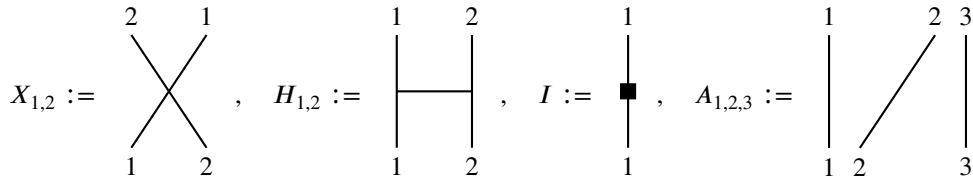

    \centering
   $ X_{1,2} := \begin{array}{cc}
          \tikz{\draw[thick] (0,0.25) -- (1,1.75); \node at (0,0) {1};\node at (0,2) {2}; \draw[thick] (1,0.25) -- (0,1.75); \node at (1,0) {2}; \node at (1,2) {1}} \end{array}, \quad 
    H_{1,2} :=\begin{array}{cc}
          \tikz{\draw[thick] (0,0.25) -- (0,1.75); \draw[thick] (0,1) -- (1,1); \node at (0,0) {1};\node at (0,2) {1}; \draw[thick] (1,0.25) -- (1,1.75); \node at (1,0) {2}; \node at (1,2) {2}}\end{array}, 
          \quad 
    I :=\begin{array}{cc}
          \tikz{\draw[thick] (0,0.25) -- (0,1.75); \node at (0,0) {1};\node at (0,2) {1}; \draw[fill] (-0.1,0.9) rectangle (0.1,1.1);} \end{array}, 
          \quad 
    A_{1,2,3} : = \begin{array}{c} \tikz{\draw[thick] (0,0.25) -- (0,1.75); \node at (0,0) {1}; \node at (0,2) {1}} \tikz{ \draw[thick] (3,-0) -- (2,-1.50); \node at (2,-2+.25) {2}; \node at (3,-0+.25) {2};} \tikz{\draw[thick] (2,0.25) -- (2,1.75); \node at (2,0) {3}; \node at (2,2) {3}} \end{array}$
    \caption{Generating morphisms of the operad $\PaRCD$.}\label{fig:Generators PaRCD}
\end{figure} 

For any two operads $\calP$ and $\calQ$ in groupoids, a morphism of operads $\calP \lrar \calQ$ is called an \textit{ equivalence}
if there is an equivalence of categories $\calP(n) \lrar \calQ(n)$ in each arity. The operad $\PaRCD$ can equivalently be defined as an operadic pullback along the map $w: \magma \lrar \ob \RCD$, that is, $w^*\RCD = \PaRCD$, where $\Hom_{\PaRCD(n)}(p,q) = \Hom_{\RCD(n)}(w(p),w(q))$ for all $p,q \in \magma(n)$ (see \cite[Section 6.1.5]{FresseBook1} for a general pullback construction of operads in groupoids). 

\begin{prop}\label{prop: categorical equivalences RCD and PaRCD}
There is an equivalence between the operads $\hRCD$ and $\hPaRCD$.
\end{prop}

\begin{proof}
We need to show that there is an equivalence of prounipotent groupoids $\hPaRCD(n) \overset{\sim}{\longrightarrow} \hRCD(n)$ for each $n$. We notice that the pullback along the map $w: \magma(n) \lrar \ast$ sends every rooted tree in $\magma$ to $\ast$, induces the following identification
\begin{equation} \label{eq: fully faithful}
    \Hom_{\PaRCD(n)}(p,q) = \Hom_{\RCD(n)}(w(p),w(q)) = \Hom_{\RCD(n)}(\ast,\ast) = \exp(\mathfrak{ft}_{n}),
\end{equation}
for any $p,q \in \magma(n)$. The map $w: \magma(n) \lrar \ast$ is clearly surjective by definition and is fully faithful by \eqref{eq: fully faithful}.
\end{proof}

The cyclic structure on $\PaRCD$, which is induced by the cyclic structure on $\RCD$. We describe the action $z^*$ on the generating morphisms $H_{1,2},I, X_{1,2}$ and $A_{1,2,3}$ of $\PaRCD$ as follows: \[ z_2^*(I) = I \in \Hom_{\hPaRCD(1)}(1,1), \quad z_3^*(H_{1,2})= -H_{1,2}-I_2 \in \Hom_{\hPaRCD(2)}(12,12),\] \[ z_3^*(X_{1,2})= X_{1,2} \in \Hom_{\hPaRCD(2)}(12,12) \quad \text{and} \quad z_4^*(A_{1,2,3}) = A^{-1}_{2,3,1}\in \Hom_{\hPaRCD(3)}((12)3,1(23)).\]


\begin{remark}
We note that the morphism $A_{1,2,3}$ is not present in $\RCD$. The operad $\PaRCD$ is a non-strict refinement of $\RCD$, $A_{1,2,3}$ only appears in $\PaRCD$. The action $z_4^*(A_{1,2,3}) = A^{-1}_{2,3,1}$ comes from the action on the associator $\alpha_{1,2,3}$ of $\PaRB$ from \cite[Section 3]{campos2019configuration}. Since $\alpha_{1,2,3}$ extends the $\Sigma_3$-action to the $\Sigma_4$-action, so does $A_{1,2,3}$, similar to \cite[Lemma 4.7]{RS2025cyclicgt}. 
\end{remark}
\section{Drinfeld Associators}\label{sec:Associators}

Let $\ff_2$ be the degree-completed free Lie algebra on two variables $x$ and $y$, and let $\hat{U}(\ff_2)= \kk\langle\langle t_{12},t_{23}\rangle\rangle$ be the corresponding non-commutative formal power series ring in two variables $x$ and $y$.  The enveloping algebra, $\hat{U}(\ff_2)$, admits a natural Hopf algebra structure, and the group-like elements of $\hat{U}(\ff_2)$ are of the form $\exp(-)$; that is, an element $s \in \hat{U}(\ff_2)$ is group-like if and only if it admits a form $s = \exp(x)$ for some $x \in \ff_2$. In \cite{Drin}, Drinfeld  defines an associator to be a particular group-like element of $\kk\langle\langle t_{12},t_{23}\rangle\rangle$ that satisfies some equations as follows.

\begin{definition}\label{def: Drin Ass}
Let $\kk$ be a field that contains $\mathbb{Q}$. A \emph{Drinfeld associator} is a pair $(\lambda,\Phi)\in\kk^{\times}\times \exp(\ff_2)$ that satisfies the following equations. 
\begin{equation}\tag{I}
\Phi(x_1,x_2)\Phi(x_2,x_1)=1,
\end{equation}
\begin{equation}\label{pentagon}\tag{P}
\Phi(t_{12}, t_{23})\Phi(t_{12}+ t_{13}, t_{24}+t_{34})\Phi(t_{23}, t_{34})=\Phi(t_{13} + t_{23}, t_{34})\Phi(t_{12}, t_{23}+ t_{24}) \quad \text{in} \quad \exp({\mathfrak{t}_4}),  
\end{equation}
\begin{equation}\label{hexagon}\tag{H}
\Phi(t_{12}, t_{23})  e^{\lambda t_{23}/2} \Phi(t_{23}, t_{31})  e^{\lambda t_{31}/2}  \Phi(t_{13}, t_{12}) e^{\lambda t_{12}/2} = 1 \qquad \text{in} \quad \exp({\mathfrak{t}_3 }) \quad \text{where} \quad t_{12}+ t_{23}+ t_{13} = 0;
\end{equation} in the complete associative algebra $\kk\langle\langle t_{12},t_{23}\rangle\rangle$.
\end{definition}

We write $\assoc_{\kk}$ for the set of Drinfeld associators.

\begin{remark}
Drinfeld showed that such associators exist with rational coefficients; however, associators with coefficients in $\mathbb{Z}$ do not exist; see \cite[Section 2, Section 5]{Drin}. There are two known explicit examples of associators: one constructed by Drinfeld is $\Phi_{\mathbf{KZ}}$ over $\mathbb{C}$ using the monodromy of the Knizhnik-Zamolodchikov equations, and the other $\Phi_{\mathbf{AT}}$ was constructed by Alekseev-Torossian in \cite{AT_deformations_flat_connections} using the integration theory of singular differential forms on semialgebraic chains related to the solutions to the Kashiwara-Vergne problem.
\end{remark}

We briefly recall the operad of parenthesized ribbon braids, denoted by $\PaRB$. $\PaRB = \{\PaRB(n)\}_{n \geq 1}$ is an operad in groupoids. For all $n$, the object set of $\PaRB(n)$ is given by $\magma(n)$. For any $p_1$ and $p_2$ in $\magma(n)$, the morphisms $$\Hom_{\PaRB(n)}(p_1,p_2) := \Hom_{\CoRB(n)}(u(p_1),u(p_2))$$ are elements of the ribbon braid group $\RB_n$ whose underlying permutation is $u(p_2)u(p_1)^{-1}$. Here $u : \magma(n) \lrar \Sigma_n$ is the forgetful map that forgets the parenthesization. The operadic composition \[\circ_i: \PaRB(n) \times \PaRB(m) \lrar \PaRB(n+m -1)\] is defined by replacing the $i^{\text{th}}$ ribbon in $\PaRB(n)$ with the ribbons in $\PaRB(m)$. One can similarly define the operad of parenthesized braids $\PaB$ by simply replacing the ribbon braids in $\PaRB$ with braids. Moreover, these two operads are related by the semidirect product $\PaRB(n) \cong \PaB(n) \rtimes \mathbb{Z}^{n}$. Any morphism of $\PaRB$ admits a (non-unique) decomposition in terms of categorical and operadic compositions of braiding $\beta_{1,2}: \mu \lrar (12)^* \mu$, associativity $\alpha_{1,2,3}: \mu \circ_1 \mu \lrar \mu \circ_2 
\mu$, and twist $\tau_{}$, a single strand with full twist. We refer to \cite[Section~6]{Boavida-Horel-Robertson} for more details. By applying the prounipotent completion functor $(-)_{\kk}$ to the operad $\PaRB$, we get the operad $\PaRB_{\kk}$ in prounipotent groupoids. The operad $\PaRB$ admits a cyclic structure \cite{campos2019configuration, RS2025cyclicgt} given by 
\[z^*(\beta_{1,2}) = \beta_{1,2} \cdot \tau, \qquad z^* (\alpha_{1,2,3}) = \alpha^{-1}_{2,3,1}, \qquad z^*(\tau) = \tau.\]


It is well known that Drinfeld associators determine object-fixing operad isomorphisms \[\hPaB \lrar \hPaCD,\] see, for example, \cite{BN, FresseBook1}. 
We write $\Assoc(\kk)$ to denote the set of all such operad isomorphisms.

The following proposition is similar to \cite[Proposition 10.2.6; Theorem 10.2.9]{FresseBook1} and \cite[Section 5]{willwacher2024cyclic}.

\begin{prop}\label{prop: maps to RCD}
The operad equivalence $f: \hPaRB \rightarrow \hRCD$ is uniquely determined by a scalar parameter $\lambda\in\kk^{\times}$ and a group-like element of the complete tensor algebra on two generators $\Phi(x,y)\in T(x,y)$ satisfying the unit, involution, hexagon, and pentagon
relations such that:
\[f(\tau)=e^{\frac{\lambda}{2} t_{11}}, \quad f(\beta) = e^{\frac{\lambda}{2} t_{12}} \quad \text{and} \quad f(\alpha) = \Phi(t_{12}, t_{23})\]
in $\RCD$. Here, $\tau$, $\beta$, and $\alpha$ are the relevant generating isomorphisms of $\PaRB$.
\end{prop}

The Lemma 3.2 of \cite{willwacher2024cyclic} implies that the assignment of Proposition~\ref{prop: maps to RCD} determines a map of cyclic operads. 

\begin{lemma}
The operad equivalence $f:\PaRB \rightarrow \hRCD$ given by the pair $(\lambda, \Phi) \in \kk^{\times}\times \exp(\ff_2)$ uniquely lifts to an equivalence of cyclic operads. 
\end{lemma}

\begin{proof}
It suffices to check the cyclic compatibility on the generators of $\PaRB$.
    
For braiding: $z^*f(\beta_{1,2}) = z^*(e^{\frac{\lambda}{2} t_{12}}) = e^{\frac{\lambda}{2} z^* t_{12}} = e^{\frac{\lambda}{2} (-t_{12}- t_{22})} = f(\beta^{-1}_{1,2} \cdot \tau_2) = f(z^* \beta_{1,2}).$

For twist: $z^*f(\tau_{1}) = z^*(e^{\frac{\lambda}{2} t_{11}}) = e^{\frac{\lambda}{2} z^* t_{11}} = e^{\frac{\lambda}{2} t_{11}}  = f(z^* \tau_{1})$.
    
 For associativity: $f(z^*\alpha_{1,2,3}) = f(\alpha_{2,3,1}^{-1}) = \Phi(t_{23},t_{31})^{-1} =\Phi(t_{31}, t_{23})$, and \[z^*f(\alpha_{1,2,3}) = z^* \Phi(t_{12}, t_{23}) = \Phi(-t_{12}-t_{22}-t_{32} , t_{23}) = \Phi(t_{31}-t_{22}, t_{23}) = \Phi(t_{31}, t_{23}), \] where the second equality uses the cyclic action on $t_{12}$ seen as an element in $\mathfrak{ft}_3$; the third uses the central identity $t_{12}+t_{23}+t_{31} = 0$ and $t_{ij} = t_{ji}$ for all $i,j$. The last equality uses the fact that $t_{22}$ is central, and for any central element $z$, $\Phi(x+z,y) = \Phi(x,y)$.
\end{proof}

\begin{lemma}\label{lemma: maps PaRB to PaRCD}
A morphism of operads $f:\PaRB \rightarrow \hPaRCD$ is given by the pair $(\lambda, \Phi) \in \kk^{\times}\times \exp(\ff_2)$,
in particular
\[f(\tau)= e^{\frac{\lambda}{2} t_{11}} \cdot I_{1}, \quad f(\beta) = e^{\frac{\lambda}{2} t_{12}} \cdot X_{1,2} \quad \text{and} \quad f(\alpha) = \Phi(t_{12}, t_{23})\cdot A_{1,2,3},\] lifts to a morphism of cyclic operads. 
\end{lemma}

\begin{proof}
We need to check the cyclic compatibility of the map $f$. The actions $z^*_2 \cdot I_{1} = I_{1}$ and $z^*_2 \cdot t_{11} = t_{11}$ imply $z^*_2 \cdot f(\tau) = e^{\frac{\lambda}{2} t_{11}} \cdot I_{1} = f(z^*_2 \cdot \tau)$. Similarly, $z^*_3 \cdot f(\beta_{1,2}) = e^{\frac{\lambda}{2} (-t_{12}-t_{22})} \cdot X_{1,2}$ and $ f(z^*_3 \cdot \beta_{1,2}) = f(\beta_{1,2}^{-1} \tau_{2}^{-1}) = e^{\frac{\lambda}{2} (-t_{12}-t_{22})} \cdot X_{1,2}$. 

To check the cyclic action on the associator $\alpha_{1,2,3}$, we first note that the generators $t_{12}$ and $t_{23}$ of $\mathfrak{ft}_3$ are written operadically as \[t_{12} = \id_2\circ_1 H_{1,2} = H_{1,2}, \quad \text{and} \quad t_{23}  = A_{1,2,3}(\id_2 \circ_2 H_{1,2})A_{1,2,3}^{-1} = A_{1,2,3} H_{2,3} A_{1,2,3}^{-1}.\]
We check that $z^*_4 \cdot t_{12} = z^*_4 \cdot (\id_2\circ_1 H_{1,2}) = -(\id_2\circ_1 H_{1,2}) - (\id_2\circ_2 I) - (\id_2\circ_2 H_{2,1})$ and $z^*_4 \cdot t_{23} = A^{-1}_{2,3,1}(\id_2\circ_2 H_{1,2}) A_{2,3,1}$. Now we have the following \[f(z^*_4 \cdot \alpha_{1,2,3}) = f(\alpha_{2,3,1}^{-1}) = A^{-1}_{2,3,1} \cdot \Phi(t_{23},t_{31})^{-1} = A^{-1}_{2,3,1} \cdot \Phi(t_{31}, t_{23}),\] Since $z^*_4 \cdot \Phi(t_{12}, t_{23}) = \Phi(t_{31}, t_{23})$, we get $z^*_4 \cdot f(\alpha_{1,2,3}) = \Phi(t_{31}, t_{23})\cdot A_{2,3,1}^{-1}$, which reduces to \[\Phi( t_{31} , A^{-1}_{2,3,1} t_{23} A_{2,3,1} )\cdot A^{-1}_{2,3,1} = A^{-1}_{2,3,1} \cdot \Phi(t_{31},t_{23})^{-1}\cdot A_{2,3,1}\cdot A^{-1}_{2,3,1}  = A^{-1}_{2,3,1} \cdot \Phi(t_{31}, t_{23}).\]
\end{proof}

Let $\Op^{+}(\calO,\calP)$ be the set of object-fixing operad isomorphisms $f:\calO\to\calP$, and let $\Cyc^{+}(\calO,\calP)$ be the set of object-fixing cyclic operad isomorphisms. Lemma~\ref{lemma: maps PaRB to PaRCD} implies the following isomorphisms. 
\begin{equation}\label{eq: iso between cyc associator and associator}
\Op^{+}(\PaRB, \hPaRCD) \cong  \Op^{+}(\hPaRB, \hPaRCD) \cong \Cyc^{+}(\hPaRB^{\cyc}, \hPaRCD^{\cyc}).
\end{equation}

\begin{prop}\label{prop:bijection framed Assoc and Assoc}
There is a bijection between the sets $\mathsf{Assoc}_{\kk}$ and $\Op^{+}(\hPaRB, \hPaRCD)$. 
\end{prop}

\begin{proof}
The universal property of prounipotent completion imply that every map $\PaRB\rightarrow \hPaRCD$ is the unique extension of an operad map $\hPaRB\rightarrow \hPaRCD$.  We can therefore apply \cite[Lemma 7.4 ]{Boavida-Horel-Robertson} to deduce that an operad map $F: \hPaRB \lrar \hPaRCD$ defines an operad map $\bar{F}:\hPaB\rightarrow \hPaCD$ if we set $F(\tau) =0$. The assignment $F\mapsto \bar{F}$ defines a map $\Iso_{0}(\hPaRB, \hPaRCD)\rightarrow \Iso_{0}(\hPaB, \hPaCD)$ that admits a section.  Indeed, an operad map $\bar{F}: \hPaB\lrar \hPaCD$ can be extended to an operad map $F: \hPaRB \lrar \hPaRCD$ via the values \[F(\beta_{1,2}) =\bar{F}(\beta_{1,2}) = e^{\mu t_{12}/2} X_{1,2}, \qquad F(\alpha_{1,2,3}) =\bar{F}(\alpha_{1,2,3}) = f(t_{12}, t_{23}) \A_{1,2,3}, \quad \text{and} \quad F(\tau) = e^{\lambda t_{11}/2} I, \] for some $\lambda\in \kk^{\times}$. Since we require that $F(\tau)\in\hPaCD(1)$ respects the ribbon twist axiom, we have that
\begin{multline}
    F(\tau \circ_1 \id_{1,2}) = e^{\frac{\lambda t_{11}}{2}} \circ_1 e^{0} = e^{\frac{\lambda (t_{11} \circ_1 0)}{2}} = e^{\frac{\lambda(t_{11}+t_{22}+ 2 t_{12})}{2}} =\\
    F(\beta_{1,2} \cdot \beta_{2,1}\cdot (\id_{1,2} \circ_1 \tau) \cdot (\id_{1,2} \circ_2 \tau)) = e^{\frac{\mu t_{12}}{2}} \cdot e^{\frac{\mu t_{12}}{2}} \cdot e^{\frac{\lambda (0 \circ_1 t_{11})}{2}} \cdot e^{\frac{\lambda (0 \circ_2 t_{22})}{2}} = e^{\mu t_{12} + \frac{\lambda(t_{11}+ t_{22})}{2}}.
\end{multline} Since the elements $t_{ii}$ are central in $\mathfrak{ft}_2$, the map $F$ will only satisfy the ribbon twist axiom if $\lambda = \mu$. It follows that every framed $\kk$-associator $F:\hPaRB \lrar \hPaRCD$ reduces to the data of a pair $(\mu, f)\in \kk^{\times} \times \exp(\ft_3)$ which satisfies the defining equations of the $\kk$-associator $\bar{F}:\hPaB \lrar \hPaCD$. 
\end{proof}

\begin{remark}
The Proposition~\ref{prop:bijection framed Assoc and Assoc} differs from \cite[Proposition 5.3]{gonzalez2018contributions} in the presentation of the framed Drinfeld-Kohno Lie algebra. By replacing $t_{ii}/2$ with $t_{ii}$, one obtains Gonzalez's result on the bijection between framed $\kk$-associators and $\kk$-associators. This is due to the difference between our presentation of the framed Drinfeld-Kohno Lie algebra and the presentation by Ševera~\cite{severa2009formality}. 
\end{remark}

\begin{prop}\label{prop: bijection of cyclic associators}
There is a bijection between the sets $\assoc_{\kk}$ and $\Cyc^{+}(\hPaRB^{\cyc}, \hPaRCD^{\cyc})$.
\end{prop}

\begin{proof}
It follows from the following bijections
\begin{equation}
\assoc_{\kk} \leftrightarrow \Assoc_{\kk} \leftrightarrow \Op^{+}(\hPaB, \hPaCD) \leftrightarrow \Op^{+}(\hPaRB, \hPaRCD) \leftrightarrow \Cyc^{+}(\hPaRB^{\cyc}, \hPaRCD^{\cyc}),
\end{equation}
Here, the second and third bijections follow from Proposition~\ref{prop:bijection framed Assoc and Assoc}, and the last follows from Lemma~\ref{lemma: maps PaRB to PaRCD}
\end{proof}

\section{The graded Grothendieck--Teichmüller group and cyclic automorphisms}\label{sec:GRT}

In this section we identify the graded Grothendieck--Teichmüller group with the group of object-fixing automorphisms of the cyclic operad $\hPaRCD^{\cyc}$. This gives the cyclic operadic counterpart of the classical identification of $\GRT_{\kk}$ with automorphisms of parenthesized chord diagrams. The main point is that, unlike the case of $\PaB$ and $\PaRB$, one cannot appeal to a known presentation of $\PaCD$ or $\PaRCD$ to deduce cyclic compatibility formally. Instead, we verify directly that the defining relations of $\GRT_{\kk}$ are preserved by the cyclic action. 


Let $\ff_2$ denote the degree-completed free Lie algebra on the generators $x$ and $y$. For any complete filtered algebra $M$ and any pair of elements $a,b \in M$, the universal property of $\hat U(\ff_2)$ determines a continuous algebra morphism
\[
\gamma_{a,b} \colon \hat U(\ff_2)\to M
\]
sending $x\mapsto a$ and $y\mapsto b$. For $f\in \hat U(\ff_2)$, we write $f(a,b):=\gamma_{a,b}(f)\in M.$

\begin{definition}[\cite{Drin}]\label{Def:GRT}
The proalgebraic graded Grothendieck--Teichmüller group $\GRT_{\kk}$ is the semi-direct product $\GRT_1 \rtimes \kk^{\times}$, where the set $\GRT_1$ consists of elements $\Phi \in \exp({\ff_2}) \subset \exp({\ft_3})$ satisfying the following relations,
\begin{align}
    \Phi(x, y) &= \Phi(y, x)^{-1}, \ in \ \exp({\ft_3}) \tag{I}\label{GRT:I}, \\
    \Phi(x, y)\Phi(y, z)\Phi(z, x) &= 1, \ \text{whenever } x + y + z = 0, \ in \ \exp({\ft_3})  \tag{H} \label{GRT:H},\\
   \Phi(t_{12}, t_{23}) \Phi(t_{12}+t_{13}, t_{24}+t_{34})\Phi(t_{23}, t_{34}) &= \Phi(t_{13}+t_{23}, t_{34}) \Phi(t_{12}, t_{23}+t_{24}), \ in \ \exp({\ft_4})\tag{P}\label{GRT:P}.
\end{align}
\end{definition}

The group law on $\GRT_1$ is given by, for any two $\Phi_1$ and $\Phi_2$, by 
\begin{equation}
    ( \Phi_1 \ast \Phi_2)(x,y) = \Phi_1(\Phi_2(x,y)^{-1}x\Phi_2(x,y),y) \Phi_2(x,y).
\end{equation} The action of $\kk^{\times}$ on $\GRT_1$, given by $\Phi(\lambda^{-1}x,\lambda^{-1}y)$ for $\lambda \in \kk^{\times}$, induces a semidirect product $\GRT_1 \rtimes \kk^{\times}$.


An element $(\lambda, \Phi) \in \GRT_{\kk}$ uniquely determines an operad morphism $g:\PaRCD_{\kk} \rightarrow \hPaRCD$ defined by
\begin{equation}\label{eq: assignment of PaRCD to PaRCD}
    g(I_{1})= \lambda I_{1}, \quad g(H_{1,2}) = \lambda H_{1,2}, \quad g(X_{1,2}) = X_{1,2} \quad \text{and} \quad g(A_{1,2,3}) = \Phi(t_{12}, t_{23})\cdot A_{1,2,3}.
\end{equation}
This assignment gives a group isomorphism $\GRT_{\kk} \cong \Aut^{+}_{\Op}(\hPaRCD)$, which follows from \cite[Theorem 10.3.10]{FresseBook1} applied in the framed setting and from the observation that the framing generators $t_{ii}$ are central elements in $\mathfrak{ft}_{n}$ and do not add any new relation between the automorphisms.

Unlike $\PaB$ and $\PaRB$, there is no known explicit presentation of $\PaCD$ by generators and relations, although $\PaCD$ is known to be generated, as an operad in the category of small categories enriched in coassociative coalgebras, by the binary object $12\in \PaCD(2)$ together with the morphisms $A$, $X$, and $H$; see \cite[Remark 2.24]{calaque2024associatorsoperadicpointview} or \cite[Section 10.2]{FresseBook1}. The same issue persists in the framed setting for $\PaRCD$. For this reason, rather than deducing the cyclic compatibility formally from a presentation, we directly verify that the defining relations of $\GRT_{\kk}$ are preserved by the cyclic action $z^*$.

\begin{prop}\label{prop: cyclic GRT on (H) and (I)}
The involution \eqref{GRT:I} and hexagon \eqref{GRT:H} relations of $\GRT_{\kk}$ are preserved under the cyclic action $z^*$.
\end{prop}

\begin{proof}
For the involution relation \eqref{GRT:I}, we compute \[z^* \Phi(t_{12},t_{23}) = \Phi(-t_{12}-t_{22}-t_{32},t_{23}) = \Phi(-t_{12}-t_{32},t_{23}) = \Phi(t_{31},t_{23}) = \Phi(t_{23},t_{31})^{-1} = z^*\Phi(t_{23},t_{12})^{-1},\]
where we use the fact that $t_{ii}$ and $t_{12}+t_{23}+t_{31}$ are central elements of $\fft_3$. For the hexagon relation \eqref{GRT:H}, we first note the following: 
\begin{itemize}
    \item $z^* \Phi(t_{12},t_{23}) = \Phi(t_{31},t_{23})$.
    \item $z^*\Phi(t_{13},t_{12}) = \Phi(-t_{13}-t_{23}-t_{33}, t_{31}) = \Phi(-t_{13}-t_{23}, t_{31}) = \Phi(t_{12}, t_{31})$.
    \item $z^* \Phi(t_{23},t_{13}) = \Phi(t_{23},-t_{13}-t_{23}-t_{33}) = \Phi(t_{23}, t_{12})$.
\end{itemize}
Now, applying the action $z^*$ on the hexagon equation $\eqref{GRT:H}$ with $x = t_{31}, y = t_{12}$ 
    \begin{equation} \label{H'}
        \Phi(t_{31}, t_{12}) \Phi(t_{23}, t_{31}) \Phi(t_{12}, t_{23}) = 1
    \end{equation}
We get
    \begin{equation*}
        \begin{aligned}
	z^* \left( \Phi(t_{31}, t_{12}) \Phi(t_{23}, t_{31}) \Phi(t_{12}, t_{23})\right) = \Phi(t_{12}, t_{13}) \Phi(t_{23}, t_{12}) \Phi(t_{31}, t_{23})
	 = \Phi(t_{13}, t_{12})^{-1} \Phi(t_{12}, t_{23})^{-1} \Phi(t_{23}, t_{31})^{-1} \\
	 = \Phi(t_{13}, t_{12})^{-1} \left(\Phi(t_{23}, t_{31})  \Phi(t_{12}, t_{23})\right)^{-1}
     =  \Phi(t_{13}, t_{12})^{-1} \Phi(t_{31},t_{12})
      =  1 =  z^*(1).
    \end{aligned}
    \end{equation*}

Here, the second equality uses $\eqref{GRT:I}$ and the fourth uses $\eqref{H'}$.
\end{proof}

To prove cyclic invariance of the pentagon relation \eqref{GRT:P}, it is convenient to replace the usual form of the pentagon by an equivalent $5$-cycle relation. This reformulation, due to Furusho \cite{Furusho_pentagon}, takes place in the completed enveloping algebra of the framed sphere braid Lie algebra $\mathfrak{fB}_5$ and is better adapted to the cyclic action.

\begin{definition}
The \emph{framed sphere braid} Lie algebra $\mathfrak{fB}_n$ is a degree completed free Lie algebra generated by symbols $\{X_{ij} = X_{ji}, 1\leq i \leq j\leq n\}$, with relations 
\begin{equation}\label{equ:spherebraid_rel}
\begin{aligned}
[X_{ij}, X_{kl}] &= 0 \quad \text{for $\{i,j\}\cap\{k,l\}=\emptyset$}, \\
[X_{ij},X_{ki}+X_{kj}] &= 0  \quad  \text{for distinct $i,j,k$}, \\ 
[X_{ij}, X_{kk}] &= 0 \quad \text{for any $i,j$, and $k$}, \\
\sum_{j = 1}^{n} X_{ij} &=  0 \quad \text{for $1 \leq i \leq n$}. 
\end{aligned}
\end{equation}    
\end{definition}

The Lie algebra $\mathfrak{fB}_n$ is a framed version of $\mathfrak{B}_n$ introduced in Ihara \cite[Section 5.3]{ihara_braids}. There is a natural surjection $ \mathfrak{ft}_n \rightarrow \mathfrak{fB}_{n+1}$ that sends $t_{ij} \mapsto X_{ij}$ for all $1 \leq i \leq j \leq n$. The surjection map induces a morphism $U(\mathfrak{ft}_n) \rightarrow U(\mathfrak{fB}_{n+1})$. 

A useful way to understand the cyclic symmetry of the pentagon is through its reformulation as a $5$-cycle relation on the moduli space $\mathcal{M}_{0,5}$ of genus zero curves with five marked points. The point is that these moduli spaces carry a natural cyclic symmetry, coming from the cyclic ordering of the marked points, and so they provide a more natural setting in which to study the pentagon relation from the perspective of cyclic operads; see Kimura-Stasheff-Voronov \cite{Voronov_1995_cyclicModuli}. As pointed out by Ihara \cite[Section 5.3]{ihara_braids}, the pentagon equation for $\GT_{\kk}$ in $\PB_4(\kk)$,
\begin{equation}\label{GT:P}
    \Phi(x_{12}, x_{23}) \Phi(x_{12}x_{13}, x_{24}x_{34}) \Phi(x_{23}, x_{34})
    =
    \Phi(x_{13}x_{23}, x_{34})  \Phi(x_{12}, x_{23}x_{24}),
\end{equation}
is equivalent to the $5$-cycle relation
\[
\Phi(x_{12},x_{23})  \Phi(x_{34},x_{45})  \Phi(x_{51},x_{12})  \Phi(x_{23},x_{34})  \Phi(x_{45},x_{51}) = 1.
\]

Similarly, the pentagon relation \eqref{GRT:P} for $\GRT_{\kk}$ is equivalent to the following $5$-cycle relation in $\hat U(\mathfrak{fB}_5)$:
\[
\Phi(X_{12}, X_{23})  \Phi(X_{34}, X_{45})  \Phi(X_{51}, X_{12})  \Phi(X_{23}, X_{34})  \Phi(X_{45}, X_{51}) = 1, 
\]
see \cite[Lemma 5]{Furusho_pentagon}. It is this reformulation that is best adapted to the cyclic action, and it is the version we use below to prove invariance of the pentagon relation under $z^*$.


\begin{prop}\label{prop: cyclic GRT on (P)}
The pentagon relation \eqref{GRT:P} of $\GRT_{\kk}$ is preserved under the cyclic action $z^*$.
\end{prop}

\begin{proof}
Using \cite[Lemma 5]{Furusho_pentagon}, the pentagon equation \eqref{GRT:P} is equivalent to the following in $\hat{U}(\mathfrak{fB}_5)$:
\begin{equation}
\Phi(X_{12}, X_{23})  \Phi(X_{34}, X_{45})  \Phi(X_{51}, X_{12})  \Phi(X_{23}, X_{34})  \Phi(X_{45}, X_{51}) = 1
\end{equation}
Equivalently, using $\Phi(x,y)^{-1} = \Phi(y,x)$                   
\begin{equation}\label{eqn: GRT equivalent (P)}
   \Phi(X_{12}, X_{51})  \Phi(X_{45},X_{34}) = \Phi(X_{23}, X_{34})  \Phi(X_{45}, X_{51})  \Phi(X_{12}, X_{23})  
\end{equation}
Applying the cyclic action $z^*$ using the formula \eqref{cyclic structure on RCD} on the left side of \eqref{GRT:P}, we get 
\[\Phi(-t_{12}-t_{22}-t_{32}-t_{42},t_{23}+t_{24})  \Phi(-t_{13}-t_{23}-t_{33}-t_{43} +t_{23},t_{34}),\]
which is equivalent to \[\Phi(-X_{12}-X_{22}-X_{32}-X_{42},X_{23}+X_{24})  \Phi(-X_{13}-X_{23}-X_{33}-X_{43}+X_{23},X_{34})\] in $\mathfrak{fB}_5$. Now, using the relations $\sum_{j=1}^{5} X_{ij} = 0$ gives the following
\begin{equation}\label{eqn: LHS of (P)}
    \Phi(X_{52},X_{23}+X_{24})  \Phi(X_{23}+X_{53},X_{34}).
\end{equation}

We note that the first term of $\eqref{eqn: LHS of (P)}$ is \[\Phi(X_{52},X_{23}+X_{24}) = \Phi(X_{52},-X_{21}- X_{25}-X_{22}) = \Phi(X_{52},X_{15})\] because $[X_{52}, X_{51}+X_{52}+X_{12}] = 0 = [X_{15}, X_{15} + X_{25} + X_{12}]$ and $[X_{15},X_{22}] = 0 = [X_{52},X_{22}]$. 

Similarly, the second term of $\eqref{eqn: LHS of (P)}$ is
\[\Phi(X_{23}+X_{53},X_{34}) = \Phi(-X_{13}- X_{33} - X_{43},X_{34}) = \Phi(X_{41}, X_{34})\]
because $[X_{41}, X_{13}+X_{33}+X_{34}+X_{14}] = 0 = [X_{41}, X_{13}+X_{33}+X_{34}+X_{14}]$.

Finally, the cyclic action on the left side of $\eqref{eqn: GRT equivalent (P)}$ reduces to 
\begin{equation}\label{eqn: final LHS of (P)}
    \Phi(X_{52},X_{15})  \Phi(X_{41}, X_{34}).
\end{equation}

Now, we apply the cyclic action on the right side of \eqref{GRT:P}. We get 
\[\Phi(t_{23},t_{34})  \Phi(-t_{12}-t_{22}-t_{32}-t_{42}-t_{13}-t_{23}-t_{33}-t_{43}, t_{24}+t_{34})  \Phi(-t_{12}-t_{22}-t_{32}-t_{42},t_{23})\]
Similarly, taking the image of the last equation in $\mathfrak{fB}_5$ and using the relation $\sum_{j=1}^{5} X_{ij} = 0$ gives 
\begin{equation}\label{eqn: RHS of (P)}
    \Phi(X_{23},X_{34})  \Phi(X_{52}+X_{53},X_{24}+X_{34})  \Phi(X_{52},X_{23}). 
\end{equation} 

We now observe that the middle term of \eqref{eqn: RHS of (P)} reduces to 
\begin{align*}
  \Phi(X_{52}+X_{53},X_{24}+X_{34}) &\overset{(1)}{=} \Phi(-X_{51} - X_{54} - X_{55}, -X_{14} - X_{44} - X_{54}) \\ 
    &\overset{(2)}{=} \Phi(-X_{51} - X_{54} - X_{55}, X_{51}) \\ 
    &\overset{(3)}{=}  \Phi(X_{41}, X_{15}).  
\end{align*}

Here, the equality $(1)$ uses the relation from $\mathfrak{fB}_5$, $(2)$ follows from $[X_{51}, X_{15}+ X_{41} + X_{54} + X_{44}] = 0 = [X_{54} +X_{15}, X_{15}+ X_{41}+ X_{54} + + X_{44}]$ and for the last equality $(3)$, we use the vanishing Lie brackets $[X_{15}, X_{41}+X_{54}+X_{15} + X_{55}] = 0 =  [X_{41}, X_{41} + X_{54} + X_{15}+ X_{55}]$.

Finally, the cyclic action on the right-hand side of $\eqref{eqn: GRT equivalent (P)}$ is the following
\begin{equation}\label{eqn:  final RHS of (P)}
    \Phi(X_{23},X_{34})  \Phi(X_{41},X_{15})  \Phi(X_{52},X_{23}). 
\end{equation}

Using the following permuted pentagon obtained by applying permutation (1\,5) on \eqref{eqn: GRT equivalent (P)}, 
\begin{equation}
   \Phi(X_{52}, X_{15})  \Phi(X_{41},X_{34}) = \Phi(X_{23}, X_{34})  \Phi(X_{41}, X_{15})  \Phi(X_{52}, X_{23})  
\end{equation}

Therefore, equations $\eqref{eqn: final LHS of (P)}$ and $\eqref{eqn:  final RHS of (P)}$ are equal.
\end{proof}

\begin{thm}\label{thm: GRT is cyclic PaRCD}
    $\GRT_{\kk} \cong \Aut_{\Cyc}^{+}(\hPaRCD^{\cyc})$.
\end{thm}

\begin{proof}
An element $(\lambda, \Phi) \in \GRT_{\kk}$ uniquely determines an operad morphism $g:\PaRCD_{\kk} \rightarrow \hPaRCD$ defined by the assignment ~\eqref{eq: assignment of PaRCD to PaRCD}. We need to show that $z^*(g(-)) = g(z^*(-))$ in the generating morphisms $I_{1}, H_{1,2}, X_{1,2}$ and $A_{1,2,3}$ of $\PaRCD_{\kk}$ and for every $\Phi \in \GRT_{\kk}$, the defining relations \eqref{GRT:H}, \eqref{GRT:I}, and \eqref{GRT:P} of $\GRT_{\kk}$ preserve the cyclic action.  

The first step is as follows. It is clear that $z^*(g(X_{1,2})) = g(z^*(X_{1,2}))$. For $I_1$, we have $z_2^* \cdot g(I_{1}) = z_2^* \cdot \lambda I_{1} = \lambda I_{1} = g(z_2^* \cdot I_{1})$. For $H_{1,2}$, we have $z_2^* \cdot g(H_{1,2}) = z_2^* \cdot \lambda H_{1,2} =  \lambda (-H_{1,2}-I_2) =  g(-H_{1,2} -I_{2}) = g(z_2^* \cdot H_{1,2})$. For $A_{1,2,3}$, $z_4^* \cdot g(A_{1,2,3}) =  g(z_4^* \cdot A_{1,2,3})$ is verbatim to the arguments of $z_4^* \cdot f(\alpha_{1,2,3}) =  f(z_4^* \cdot \alpha_{1,2,3})$ from Lemma~\ref{lemma: maps PaRB to PaRCD}.

Proposition~\ref{prop: cyclic GRT on (H) and (I)} and Proposition~\ref{prop: cyclic GRT on (P)} together show that the required relations are compatible $z^*(g(-)) = g(z^*(-))$ under cyclic action $z^*$. 

\end{proof}

\begin{remark}
The spherical residue condition $\sum_{i=0}^n t_{ij} = 0$ for all $0 \leq j \leq n$ of $\fft_{n}$ is used by Willwacher \cite{willwacher2024cyclic} to compute the homotopy automorphism space of the cyclic BV Hopf cooperad, in particular it was shown that  $\GRT_{\kk} \cong \HoAut_{\mathrm{CycHopfOp^{c}}}(\mathsf{BV}^{c})$, where the Batalin-Vilkovisky cooperad $\mathsf{BV}^{c}$ is seen as a cyclic dg Hopf cooperad. Theorem~\ref{thm: GRT is cyclic PaRCD} can be seen as a strict $0$-truncation of Willwacher's result \cite[Corollary 1.5]{willwacher2024cyclic}. Therefore, these two approaches are complementary where one uses homotopy automorphism and the other discrete automorphisms. Theorem~\ref{thm: GRT is cyclic PaRCD} is also the graded version of the main result in the prounipotent setting of Robertson and the author \cite{RS2025cyclicgt}.
\end{remark}

\begin{remark}
Furusho showed in \cite{Furusho_pentagon} that the pentagon equation \eqref{GRT:P} implies the involution \eqref{GRT:I} and hexagon equations \eqref{GRT:H} of Definition~\ref{Def:GRT}. Accordingly, for the purpose of proving compatibility with the cyclic symmetry, it would in principle suffice to establish Proposition~\ref{prop: cyclic GRT on (P)}. We record Proposition~\ref{prop: cyclic GRT on (H) and (I)} as well, since the argument is short and makes the cyclic action more transparent.

\end{remark}

Since $\GT_{\kk}$ acts freely and transitively from the left on the set $\Assoc_{\kk}$ of Drinfeld associators, while $\GRT_{\kk}$ acts freely and transitively from the right, and these actions commute. Thus, the triple $(\GT_{\kk},\Assoc_{\kk},\GRT_{\kk})$ is a bitorsor.

\begin{prop}\label{prop: bitorsor triple}
The triple
\[
(\Aut^{+}_{\Cyc}(\hPaRB^{\cyc}), \Iso^{+}(\hPaRB^{\cyc}, \hPaRCD^{\cyc}), \Aut^{+}_{\Cyc}(\hPaRCD^{\cyc}))
\]
is a bitorsor.
\end{prop}

\begin{proof}
The classical operadic description of Drinfeld associators yields the bitorsor
\begin{equation}\label{eqn:bitorser triple of operads}
(\Aut^{+}_{\Op}(\hPaRB), \Iso^{+}(\hPaRB, \hPaRCD), \Aut^{+}_{\Op}(\hPaRCD));
\end{equation}
see, for example, \cite{FresseBook1}. By \cite[Theorem 7.7]{RS2025cyclicgt}, Proposition~\ref{prop: bijection of cyclic associators}, and Theorem~\ref{thm: GRT is cyclic PaRCD}, the three terms in \eqref{eqn:bitorser triple of operads} identify, respectively, with
\[
\Aut^{+}_{\Cyc}(\hPaRB^{\cyc}), \qquad
\Iso^{+}(\hPaRB^{\cyc}, \hPaRCD^{\cyc}), \qquad
\Aut^{+}_{\Cyc}(\hPaRCD^{\cyc}).
\]
Transporting the left and right actions along these identifications gives the claimed bitorsor structure.
\end{proof}

Equivalently, the standard bitorsor $(\GT_{\kk},\Assoc_{\kk},\GRT_{\kk})$ is identified with the cyclic-operadic bitorsor
\[
(\Aut^{+}_{\Cyc}(\hPaRB^{\cyc}), \Iso^{+}(\hPaRB^{\cyc}, \hPaRCD^{\cyc}), \Aut^{+}_{\Cyc}(\hPaRCD^{\cyc})).
\]


\section{Action on Chord Diagrams} \label{sec: Tangles and Chord diagrams}

We conclude with an application of the main theorem to chord diagrams. Chord diagram categories provide the natural target (the associated graded) in which one computes universal finite type invariants of braids, tangles, and related objects; see, for example, Kassel--Turaev \cite{KT_Tangles}. The main result (Theorem~\ref{prop: GRT action on metric prop}) of this section is that the cyclic operadic action of $\GRT_{\kk}$ on $\hPaRCD^{\cyc}$ induces an action on the corresponding category of parenthesized chord diagrams with self-dual objects.


\subsection{Infinitesimal Symmetric Monoidal Categories}
Infinitesimal symmetric monoidal categories encode the first-order behaviour of a deformation of a symmetric monoidal category into a braided one. They provide the linearized framework in which chord diagram categories naturally arise, and hence form the appropriate algebraic background for universal finite type invariants.


Let $(\bS,\otimes,\mathbb{1},c)$ be a strict symmetric $\kk$-linear category and let $\bS[[\hbar]]$ be the symmetric monoidal category which has the same objects as those of $\bS$ but whose morphism sets are defined by \[ \Hom_{\bS[[\hbar]]}(X,Y) \coloneqq \Hom_{\bS}(X,Y) \otimes_{\kk} \kk[[\hbar]], \] for all objects $X, Y \in \bS$. That is, $\bS[[\hbar]]$ is the category obtained by extending the morphisms of $\bS$ to formal power series in a deformation parameter $\hbar$. An \emph{infinitesimal braiding} on $\bS$ is a natural family of endomorphisms
\[
t_{X,Y}\colon X\otimes Y\to X\otimes Y
\]
such that
\[
t_{X,Y}=c_{X,Y}\circ t_{Y,X}\circ c_{X,Y}^{-1}
\]
and
\[
t_{X\otimes Y,Z}
=
(\id_X\otimes t_{Y,Z})
+
(\id_X\otimes c_{Y,Z})\circ (t_{X,Z}\otimes \id_Y)\circ (\id_X\otimes c_{Y,Z})^{-1}.
\]
These conditions ensure that $t$ is the first-order part of a braiding on the formal deformation $\bS[[\hbar]]$, namely
\[
c_{X,Y}\Bigl(\id_{X\otimes Y}+\frac{\hbar}{2}t_{X,Y}\Bigr);
\]
see \cite{Cartier1992Vessiliev_Invarients}. A \emph{strict infinitesimal symmetric monoidal category} is a strict symmetric $\kk$-linear category equipped with an infinitesimal braiding. Equivalent reformulations, as well as the non-strict analog, may be found in \cite[Proposition.~B.1]{KT_Tangles} and \cite[Section~10.3.3]{FresseBook1}. We say that $\bS$ has (left) \emph{duals} if, for every object $X \in \bS$, there exists an object $X^*$ together with morphisms
\[
b_X : \mathbb{1} \to X \otimes X^*, \qquad d_X : X^* \otimes X \to \mathbb{1}
\]
such that
\begin{equation}\label{eq: left dual zig zag equations}
(b_X \otimes \id_X) (\id_X \otimes d_X)  = \id_X, \qquad (\id_{X^*} \otimes b_X) (d_X \otimes \id_{X^*})  = \id_{X^*}.
\end{equation}
For any morphism $f: X \lrar Y$ in $\bS$ with duals, the transpose of $f$ is a morphism $f^*: Y^* \lrar X^*$ given by $f^* = (\id_{Y^*}\otimes b_X)(\id_{Y^*}\otimes f\otimes\id_{X^*}) (d_Y\otimes\id_{X^*})$ determined by the duality pairing $(b,d)$ (cf. \cite[Chapter XIV.2]{kassel_quantum_1995}).

\subsection{The category $\bA(\kk)$}
The category $\bA(\kk)$ may be viewed as the infinitesimal symmetric monoidal analog of the tangle category, with morphisms given by chord diagrams; see \cite[Section 5.3]{KT_Tangles}. Its objects are parenthesized words in the signed set $\{+,-\}$ ( the free magma on $\{+,-\}$). For objects $p,q\in \bA(\kk)$, the morphism space $\Hom_{\bA(\kk)}(p,q)$ is the $\kk$-vector space spanned by chord diagrams from $p$ to $q$ with the corresponding parenthesized boundary data, modulo the $4T$ relation (Figure~\ref{fig:4T relation}):
\[
\Hom_{\bA(\kk)}(p,q)
=
\mathrm{Span}_{\kk}\{\text{parenthesized chord diagrams from $p$ to $q$}\}/4T.
\]

Composition is given by vertical stacking and the tensor product by concatenation on objects and horizontal juxtaposition on morphisms. The unit object is the empty word $\emptyset$. For each word $p$, the infinitesimal braiding is given by endomorphisms
\[
t^p_{ij}\in \Hom_{\bA(\kk)}(p,p),
\]
represented by a single chord joining the $i$th and $j$th strands of the identity diagram on $p$. Duality is given by morphisms
\[
b_p:\emptyset\to p\otimes p^*,
\qquad
d_p:p\otimes p^*\to \emptyset,
\]
where $p^*$ is the reversed word with all signs changed. In this way, $\bA(\kk)$ becomes an infinitesimal symmetric monoidal category with duals.


\begin{figure}[ht!]
    \centering
$\begin{array}{cc}
    \tikz{\draw[thick] (0,0.25) -- (0,1.75); \node at (0,0) {i}; \draw[thick] (1,0.25) -- (1,1.75); \node at (1,0) {j};  \draw[thick] (2,0.25) -- (2,1.75); \node at (2,0) {k}; \draw[thick, dotted] (0,0.75) -- (1,0.75); \draw[thick, dotted] (0,1.4) -- (2,1.4);} 
\end{array} \, + \,
\begin{array}{cc}
     \tikz{\draw[thick] (0,0.25) -- (0,1.75); \node at (0,0) {i}; \draw[thick] (1,0.25) -- (1,1.75); \node at (1,0) {j};  \draw[thick] (2,0.25) -- (2,1.75); \node at (2,0) {k}; \draw[thick, dotted] (0,0.75) -- (1,0.75); \draw[thick, dotted] (1,1.4) -- (2,1.4);} 
\end{array} \, - \, 
\begin{array}{cc}
     \tikz{\draw[thick] (0,0.25) -- (0,1.75); \node at (0,0) {i}; \draw[thick] (1,0.25) -- (1,1.75); \node at (1,0) {j};  \draw[thick] (2,0.25) -- (2,1.75); \node at (2,0) {k}; \draw[thick, dotted] (0,1.45) -- (1,1.45); \draw[thick, dotted] (1,0.75) -- (2,0.75);} 
\end{array} \, - \, 
\begin{array}{cc}
    \tikz{\draw[thick] (0,0.25) -- (0,1.75); \node at (0,0) {i}; \draw[thick] (1,0.25) -- (1,1.75); \node at (1,0) {j};  \draw[thick] (2,0.25) -- (2,1.75); \node at (2,0) {k}; \draw[thick, dotted] (0,1.45) -- (1,1.45); \draw[thick, dotted] (0,0.75) -- (2,0.75);} 
\end{array} \, = \, 0 $
    \caption{The 4T relation}
    \label{fig:4T relation}
\end{figure}


\begin{thm}[\cite{Cartier1992Vessiliev_Invarients, KT_Tangles}]
The category $\bA(\kk)$ is the free infinitesimal symmetric monoidal category with duals generated by a single object. \qed
\end{thm}

This means that for any infinitesimal symmetric monoidal category $\bS$ with duals and any object $X \in \bS$, there is a unique monoidal functor $G: \bA(\kk) \to \bS$ such that $G(+) = X$ and  preserves the infinitesimal braiding and duality structure:
\[
G(-) = X^*, \quad G(t_{+,+}) = t_{X,X}, \quad G(c_{+,+}) = c_{X,X}, \quad G(\alpha_{+,+,+}) = \alpha_{X,X,X}, \quad G(b_{+}) = b_X, \quad G(d_{+}) = d_X.
\]

We now introduce the operads needed to compare the category of chord diagrams with $\hPaRCD$.

We begin by defining a framed version of the operad $\hA$ from~\cite[Section~10.3.1]{FresseBook1}, which encodes infinitesimal structures on symmetric monoidal categories. The operad $\hA$ is the universal enveloping-algebra analogue of the chord diagram operad: it is defined similarly to $\hPaCD$, but before restricting to group-like elements.

\begin{definition}
For each $n\geq 1$, let $\hA(n)$ be the groupoid with one object and endomorphism algebra
\[
\Hom_{\hA(n)}(*,*)=\hat U(\mathfrak{ft}_n),
\]
the degree-completed universal enveloping algebra of the framed Drinfeld--Kohno Lie algebra $\mathfrak{ft}_n$. The collection
\[
\hA=\{\hA(n)\}_{n\geq 1}
\]
forms an operad in complete Hopf groupoids, with operadic composition induced from that of $\mathfrak{ft}_n$; see Definition~\ref{def:DK Lie algebra}.
\end{definition}

Its parenthesized refinement is obtained by pullback along the map $w\colon \magma\to \ob(\hA)$, i.e. $\hPaA:=w^*\hA.$ Equivalently, $\hPaA(n)$ has object set $\magma(n)$, and for any $p,q\in \magma(n)$,
$\Hom_{\hPaA(n)}(p,q)=\hat U(\mathfrak{ft}_n)$,
with composition given by multiplication in $\hat U(\mathfrak{ft}_n)$.

As with the operad $\PaRCD$, the operad $\hPaA$ is generated under categorical and operadic composition by the four elements
\[
\textsf{X}_{1,2}\in \Hom_{\hPaA(2)}((12),(21)), \qquad
\textsf{H}_{1,2}\in \Hom_{\hPaA(2)}((12),(12)),
\]
\[
\textsf{I}_1\in \Hom_{\hPaA(1)}(1,1), \qquad \text{and} \qquad
\textsf{A}_{1,2,3}\in \Hom_{\hPaA(3)}(((12)3),(1(23))).
\]
These correspond respectively to the symmetry, infinitesimal braiding, identity, and associativity in an infinitesimal symmetric monoidal category. The following theorem records the corresponding universal property; see \cite[Theorem~10.3.4]{FresseBook1}.


The operad $\hPaA$ is generated by the unit, symmetry, associator, and infinitesimal braiding, and the relations among these generators are precisely the coherence relations for an infinitesimal symmetric monoidal category. In particular, maps out of $\hPaA$ are determined by the corresponding structural data. The following proposition is a consequence of \cite[Theorem~10.3.4]{FresseBook1}.

\begin{prop}\label{prop: hPaA-structure-is-inf-sym}
Let $\bC$ be a complete $\kk$-linear category. Then giving a map of operads
\[
\rho\colon \hPaA\to \End(\bC)
\]
is equivalent to equipping $\bC$ with the structure of an infinitesimal symmetric monoidal category.
\end{prop}

\begin{proof}
By \cite[Theorem~10.3.4]{FresseBook1}, an operad map $\rho\colon \hPaA\to \End(\bC)$ is determined by the images of the generators $\textsf{I},$ $\textsf{X},$ $\textsf{H},$ and $\textsf{A},$ subject to the defining relations of $\hPaA$. These are exactly the unit, symmetry, infinitesimal braiding, and associator data, together with the coherence relations for an infinitesimal symmetric monoidal category. Hence such an operad map is equivalent to an infinitesimal symmetric monoidal structure on $\bC$.
\end{proof}

The operad $\hPaRCD$ is obtained from $\hPaA$ by passing from the infinitesimal braiding to its exponential. More precisely, the generators $\textsf{A}$, $\textsf{I}$, and $\textsf{X}$ are the same in both operads, while the infinitesimal braiding $\textsf{H}_{1,2}$ in $\hPaA$ is replaced by its exponential $e^{\textsf{H}_{1,2}}$, which is group-like. In this sense, $\hPaRCD$ governs the integrated, or exponentiated, form of an infinitesimal symmetric monoidal structure. 

\begin{cor}\label{cor: maps out of prounipotent PaRCD}
Let $\bC$ be a complete $\kk$-linear category. Then a morphism of operads
\[
\rho\colon \hPaRCD \to \End(\bC[[\hbar]])
\]
determines a infinitesimal symmetric monoidal structure on $\bC[[\hbar]]$.
\end{cor}

\subsection{Envelop and Metric Prop}\label{subsec: Envelop and Metric Prop}

Given an operad $\calO$, its envelop $\Env(\calO)$ is the strict symmetric monoidal category characterized by the property that symmetric monoidal functors out of $\Env(\calO)$ are equivalent to $\calO$-algebras. In much of the literature one says that the envelop, $\Env(\calO)$, is the prop associated to $\calO$.

Concretely, $\Env(\calO)$ has object set $\mathbb N$, and its morphisms are given by
\begin{equation}\label{eq:envolepe}
\Env(\calO)(m,n)
:=
\coprod_{m_1+\cdots +m_n=m}
\bigl(\calO(m_1)\otimes\cdots\otimes \calO(m_n)\bigr)\otimes_{\Sigma_{m_1}\times\cdots\times\Sigma_{m_n}}\Sigma_m,
\end{equation}
where the coproduct is taken over all decompositions of $m$ into $n$ parts. Equivalently, there is an adjunction
\[\begin{tikzcd}
\Env : \Op \arrow[r, shift left=1] & \mathbf{Prop} \arrow[l, shift left=1] : u
\end{tikzcd}
\]
with $u(\calP)(n):=\calP(n,1).$ Thus $\Env$ is left adjoint to the forgetful functor from props to operads. In particular, giving an algebra over $\calO$ is equivalent to giving an algebra over $\Env(\calO)$.

\begin{definition}
Let $\calO^{\cyc}$ be a cyclic operad, with underlying operad $\calO$. The \emph{metric prop} $\Pi(\calO^{\cyc})$ is the prop obtained from $\Env(\calO)$ by adjoining morphisms
\[
d\in \Pi(\calO^{\cyc})(2,0)
\qquad \text{and} \qquad
b\in \Pi(\calO^{\cyc})(0,2),
\]
subject to the following relations:
\begin{enumerate}
    \item $(b \otimes \id_1)(\id_1 \otimes d) = (\id_1 \otimes b) (d \otimes \id_1) = \id_1 \in \Pi(\calO^{cyc})(1,1)$. \label{eq:zig-zag}
    \item For any operation $f\in \calO(n)$, $(b\otimes \id_n)(\id_1 \otimes f \otimes \id_1) (\id_1 \otimes d) = z^* f \in \calO^{\text{cyc}}(n)$, where the $z^* f$ is viewed in $\Env(\calO)\subset \Pi(\calO^{\cyc})$.
\end{enumerate}

\end{definition}

An algebra over the metric prop $\Pi(\calO^{\cyc})$ in a symmetric monoidal category $\bE=(\bE, \otimes, \mathbb{1})$ is a strict symmetric monoidal functor
\[
F\colon \Pi(\calO^{\cyc}) \to \bE.
\]
If we write $A:=F(1)$, then the images
\[
F(d)\colon A\otimes A\to \mathbb{1}
\qquad \text{and} \qquad
F(b)\colon \mathbb{1}\to A\otimes A
\]
define a nondegenerate symmetric pairing on $A$. The invariance relation then says that, for every operation $f\in \calO(n)$, the cyclic action $z^*f$ is obtained from $f$ by transporting one input past the pairing. In other words, the cyclic structure is encoded by the transpose operation $f^*: Y^* \lrar X^*$ of a morphism $f: X \lrar Y$ in a tensor category with duals.

We will use the metric prop construction to pass from the operad $\hPaRCD$ to the corresponding tensor category, which we can compare with the chord diagram category $\bA(\kk)$. 

\begin{definition}
We define the category $\bA'(\kk)$ to be the quotient of $\bA(\kk)$ imposing the strict self-duality $(+)^* = +$.    
\end{definition}
Equivalently, $\bA'(\kk)$ is the free infinitesimal symmetric monoidal category with a single self-dual generator. From now on, we are mainly working with the category $\bA'(\kk)$. The following lemma formalizes the relationship between the cyclic operad $\hPaRCD^{\cyc}$ and the category $\bA'(\kk)$:

\begin{lemma}\label{lemma: metric prop of PaRCD}
The metric prop associated with the cyclic operad $\hPaRCD^{\cyc}$ is equivalent, as a tensor category, to the category $\bA'(\kk)$.
\end{lemma}

\begin{proof}

The algebra over the operad $\hPaRCD$ is the same as the algebra over the prop $\Env(\hPaRCD)$. We get a tensor functor $\phi: \hPaRCD \lrar \bA(\kk)$ defined by sending generators \[ \phi(1) = +, \quad \phi(H_{1,2}) = t_{+,+}, \quad \phi(I) = I_{+}, \quad \phi(X_{1,2}) = X_{+,+}, \quad \phi(A_{1,2,3}) = A'_{+,+,+}, \] where $t_{+,+}$, $I_{+}$, $X_{+,+}$ and $\alpha_{+,+,+}$ represent the infinitesimal braiding, the framed element, symmetry and associativity morphisms in $\bA(\kk)$. The coherence axioms of infinitesimal symmetric monoidal categories ensure that all relations $\PaRCD$ are preserved. Moreover, we have 
\begin{equation}\label{eq:faithful}
    \Hom_{\Env(\hPaRCD)}(p,q) \cong \Hom_{\bA(\kk)}(\phi(p),\phi(q)).
\end{equation}
The functor $\phi$ extends to a functor of props $\phi': \Pi(\hPaRCD^{\cyc}) \lrar \bA'(\kk)$ by assigning $\phi'(b) = coev$ and $\phi'(d) = ev$, where $coev: \varnothing \lrar (+,+)$ and $ev: (+,+) \lrar \varnothing$. Note that extending $\hPaRCD^{\cyc}$ to its metric prop $\Pi(\hPaRCD^{\cyc})$ forces the the identification $(+)^* = +$. Here $\phi'$ preserves the zig-zag relations \eqref{eq:zig-zag} because it is a part of the duality structure in $\Pi(\hPaRCD^{\cyc})$.
For any operation $g\in \hPaRCD^{\cyc}(n)$, the cyclic equivariance relation \[ (b\otimes \id_n)(\id_1 \otimes g \otimes \id_1) (\id_1 \otimes d) = z^* g \in \hPaRCD^{\cyc}(n)\] is the image of the transpose,
$f^*$ of $f \in \PaRB^{\cyc}_{\kk}$ under the operad isomorphism $\PaRB_{\kk} \lrar \hPaRCD$. Since the cyclic invariance is preserved in $q\bT'_{\kk}$ due to \cite[Proposition 6.5]{RS2025cyclicgt}, it is also preserved in $\bA'(\kk)$ because $q\bT'_{\kk} \cong \bA'(\kk)$.

Since $\phi'$ is faithful by \eqref{eq:faithful}, and essentially surjective because both categories have the same objects.

\end{proof}



\subsection{GRT action on Chord diagrams}\label{subsec: GRT action on the category A(K)}

There is an action of $\GT_{\kk}$ on the metric prop generated by the cyclic operad $\PaRB^{\cyc}$ \cite{RS2025cyclicgt}, and from Section~\ref{sec: cyclic chord diagrams} the cyclic operad isomorphisms $\hPaRB^{\cyc} \lrar \hPaRCD^{\cyc}$ are in bijection with Drinfeld associators. Taking these two facts together, one expects to have an explicit $\GRT_{\kk}$ action on chord diagrams with self-dual objects. In what follows, we explicitly construct an $\GRT_{\kk}$ action on the chord diagrams.

The prounipotent completion functor $(-)_{\kk}:\Op\rightarrow \Op_{\kk}$ is symmetric monoidal, and $(-)_{\kk}$ is applied on operads in each arity. Therefore, the prounipotent completion of the envelop associated with an operad $\calO$ is isomorphic to the envelop associated with the prounipotent completion of an operad $\calO_{\kk}$,
$\operatorname{Env}(\calO)_{\kk}\cong \operatorname{Env}(\calO_{_{\kk}})$.

The following results follow from $\operatorname{Env}(\calO)_{\kk}\cong \operatorname{Env}(\calO_{_{\kk}})$ and the fact that $\Env: \Op \lrar \Prop$ is a left adjoint.

\begin{prop}\label{cor: GRT action on Envelop}
There is an isomorphism of prounipotent groups \[\GRT_{\kk}\cong \Aut^{+}_{\Op}(\hPaRCD) \cong \Aut^{+}_{\Prop}(\operatorname{Env}(\hPaRCD)) \cong \Aut^{+}_{\Prop}({\operatorname{Env}(\PaRCD)_{\kk}}).\]
\end{prop}

\begin{cor}
There is an isomorphism of prounipotent metric props associated with the cyclic operad $\PaRCD^{\cyc}$ 
\[{\Pi(\PaRCD^{\cyc})}_{\kk} \cong \Pi({\PaRCD}^{\cyc}_{\kk}).\]
\end{cor}

A Drinfeld associator $\Phi$ is a group-like element of the non-commutative formal power series in the free associative algebra $ \hat{U}(\ff_2) = \kk\langle \langle X,Y\rangle\rangle$. In the category $\bA(\kk)$, the associative constraint $A': (X \otimes Y) \otimes Z \lrar X \otimes (Y \otimes Z)$ through the isomorphism $A$ of $\hPaRCD$ does not satisfy the zigzag identity; however, it satisfies the following conditions \begin{equation}\label{eq: Kontsevich Integral of unknot}
(\id \otimes b)A' (d \otimes \id) = \omega^{-1}, \end{equation} where $\omega$ is a special case of the Kontsevich integral of the unknot (cf. \cite{BN_weak_tangles, Le_Murakami_1996, KT_Tangles}) and is also known as the wheeling element in \cite{BNLT} and \cite{knot_theory_book}. It is independent of the choice of Drinfeld associators, similar to \cite[Theorem 8]{Le_Murakami_1996}.

\begin{thm}\label{prop: GRT action on metric prop}
The $\GRT_{\kk}$-action on the cyclic operad $\PaRCD^{\cyc}$ extends to the metric prop $\Pi(\hPaRCD^{\cyc})$ and hence to $\bA'(\kk)$.
\end{thm} 

\begin{proof}
The graded Grothendieck--Teichmüller group $\GRT_{\kk}$ acts on $\hPaRCD$ by sending the pair $(\lambda, \Phi) \in \GRT_{\kk} $ to a unique isomorphism $G: \hPaRCD \rightarrow \hPaRCD$ defined by the following: 
\[G(H_{1,2}) = \lambda H_{1,2}, \quad G(I_1) = \lambda I_1,\quad G(X_{1,2}) = X_{1,2}, \quad G(A) = \Phi(t_{12}, t_{23}) \cdot A_{1,2,3}. \] 
$G$ induces a unique isomorphism of the props $\Env(G): \Env(\hPaRCD) \lrar \Env(\hPaRCD)$. Hence, there is a $\GRT_{\kk}$ action on the props $\Env(\hPaRCD) \cong \Env(\hPaRCD^{\cyc}) \cong {\Env(\PaRCD^{\cyc})}_{\kk}$ through $G$, using Proposition~\ref{cor: GRT action on Envelop}. 

We define the isomorphism $G': {(\Pi(\hPaRCD^{\cyc}))} \lrar {(\Pi(\hPaRCD^{\cyc}))}$ by extending the map $G$ to the duality pairing $(d,b)$, using the similar assignment described in \cite[Proposition~6.10]{RS2025cyclicgt} as follows \[G'(b) = b , \quad G'(d) = (\id \otimes \nu) d,\]
where $ \nu = \left( (b \otimes \id) \Phi(t_{12},t_{23})(\id \otimes d) \right)^{-1}$ is an automorphism of $\Pi(\hPaRCD^{\mathrm{cyc}}(1,1)$. 
The image of the pairing $(d,b)$ must satisfy the relation \eqref{eq: Kontsevich Integral of unknot}. Applying $G'$ on the left side of \eqref{eq: Kontsevich Integral of unknot} computes as follows: \[(\id \otimes b) \Phi(t_{12},t_{23}) A (((\id \otimes \nu) d )\otimes \id).\]
Now substituting $\nu$ gives 
\begin{equation}\label{eq: 4.6}
(\id \otimes b) \Phi(t_{12},t_{23}) A (((\id \otimes \left( (b \otimes \id) \Phi(t_{12},t_{23})(\id \otimes d) \right)^{-1}) d )\otimes \id).    
\end{equation}
After canceling the term $\Phi(t_{12},t_{23})$ and using the relation $(b \otimes \id) (\id \otimes d) = \id$, the equation \eqref{eq: 4.6} simplifies to \[(\id \otimes b) A (d \otimes \id),\] which is equal to $\omega^{-1}$. Now using \cite[Theorem A]{furusho_galois_2020}), $\GT_{\kk}$--action on the unknot is trivial. Therefore, we arrive at a trivial $\GRT_{\kk}$--action on $\omega^{-1}$. This implies that the image on the right side of \eqref{eq: Kontsevich Integral of unknot} under $G'$ is also $\omega^{-1}$. We conclude that \eqref{eq: Kontsevich Integral of unknot} is preserved under $G'$.

\end{proof}

\begin{remark}
The prop map $G'$ on the duality pairing $(d,b)$, that is, as follows 
\[G'(b) = b , \quad G'(d) = (\rho \otimes \id) d, \qquad  \text{where} \qquad \rho = \left((\id \otimes b) \Phi(t_{12}, t_{23})^{-1}(d \otimes \id)\right)^{-1},\]
also gives a $\GRT_{\kk}$ action. In this case, the duality pairing must satisfy the relation $(b \otimes \id) A^{-1} (\id \otimes d) = \omega$ under $G'$. Applying $G'$ on $(b \otimes \id) A^{-1} (\id \otimes d)$ gives \[(b \otimes \id) A^{-1}  \Phi(t_{12},t_{23})^{-1} (\id \otimes ((\rho \otimes \id) d))\] That simplifies to $\omega$ as follows 
\begin{equation*}
\begin{split}
(b \otimes \id) A^{-1}  \Phi(t_{12},t_{23})^{-1} (\id \otimes ((\rho \otimes \id)  d))
  = (b \otimes \id) A^{-1}  \Phi(t_{12},t_{23})^{-1} (\id \otimes ((\left((\id \otimes b) \Phi(t_{12}, t_{23})^{-1}(d \otimes \id)\right)^{-1} \otimes \id) d))\\
= (b \otimes \id) A^{-1}  (\id \otimes ((\left((\id \otimes b)(d \otimes \id)\right)^{-1} \otimes \id) d)) = (b \otimes \id) A^{-1}  (\id \otimes ((\id^{-1} \otimes \id) d))
= (b \otimes \id) A^{-1} (\id \otimes d)) = \omega
\end{split}
\end{equation*}

The desired relation holds under $G'$ using the same arguments as Proposition~\ref{prop: GRT action on metric prop}.
\end{remark}

In the category of parenthesized framed tangles, the composition of the cap $\cap:(+{-})\to \varnothing$ and the cup $\cup:\varnothing \to(+{-})$ morphisms is identified with the $0$--framed unknot $\cap \,\circ \,\cup:= U$. By Theorem~\cite[Theorem 4.7]{KT_Tangles}, given a Drinfeld associator $\Phi$, we can define an isomorphism  $Z_{\Phi}:q\mathbf{T}\rightarrow \bA(\kk)[[\hbar]]$ from the completed category of $q$-tangles that has the property that \[
Z_{\Phi}(U)=Z_{\Phi}(\cap)\circ Z_{\Phi}(\cup)=:\omega, 
\] 
Under the equivalence $\psi: q\mathbf{T'} \rightarrow \Pi(\PaRB^{\cyc})$ from \cite[Proposition~6.5]{RS2025cyclicgt}, where $q\mathbf{T'} := q\mathbf{T}/ \{(+)^* = +\}$ , we have $\psi(\cap) = d \in \Pi(\PaRB^{\cyc})(2,0)$ and $\psi(\cup) = b \in \Pi(\PaRB^{\cyc})(0,2)$. We conclude that for every Drinfeld associator $\Phi$, we have the following commutative diagram of the equivalences of prounipotent categories
\[\begin{tikzcd} q\mathbf{T'}_{\kk}\arrow[d, "\simeq"] \arrow[r, "Z_{\Phi}"] & \bA'(\kk)\arrow[d, "\simeq"]\\
\Pi(\PaRB^{\cyc})_{\kk} \arrow[r, "\varphi_{\Phi}"] & \Pi(\PaRCD^{\cyc})_{\kk}\end{tikzcd}\] with $\varphi_{\Phi}(d\circ b)=Z_{\Phi}(U)=\omega$.

\begin{remark}\label{remark: two GRT actions}
The $\GRT_{\kk}$-action on $\bA'(\kk)$ can equivalently be seen as the image of the $\GT_{\kk}$-action on the category of framed $q$-tangles $q\bT'$ from \cite{RS2025cyclicgt} under the Kontsevich isomorphism. Moreover, Furusho \cite{furusho_galois_2020} constructed a $\GRT_{\kk}$-action on chord diagrams via the ABC-construction. The $\GRT_{\kk}$-action on chord diagrams with self dual objects from Furusho and Theorem~\ref{prop: GRT action on metric prop} agree on generators $I$, $H$ and $X$. The potential difference comes in the actions on associativity $A_{1,2,3}$ and duality pairing $(b,d)$. In Furusho's construction, the corresponding action on duality pairing is determined by C-move, which involves the wheeling element $\omega$ directly than the correction term $\nu$. Currently it is not known whether these two $\GRT_{\kk}$-actions produce the same or different automorphisms of $\bA'(\kk)$. 

\end{remark}


\bibliographystyle{plain}
\bibliography{bib.bib}

\end{document}